\documentclass[reqno]{amsart}
\usepackage{pdfsync}

\usepackage[colorlinks,
citecolor=red,
linkcolor=red,
urlcolor=magenta,
allcolors=black,
]{hyperref}

\usepackage{amsmath,amsthm,amssymb,mathrsfs,aliascnt,color,mathtools}
\usepackage{float,cite,times,color,lineno,paralist}
\usepackage{tikz,graphicx,subfigure,caption}
\usepackage{pgfplots}

\usepackage{bm}
\usepackage{breqn}
\usepackage{listing}
\usepackage[all]{xy}
\usepackage{float}
\usepackage{calc}
\usepackage{enumitem}

\numberwithin{equation}{section}
\allowdisplaybreaks

\theoremstyle{definition}
\newtheorem{theorem}{Theorem}[section]

 \theoremstyle{definition}
    \newaliascnt{lemma}{theorem}
    \newtheorem{lemma}[lemma]{Lemma}
    \aliascntresetthe{lemma}

 \theoremstyle{definition}
    \newaliascnt{corollary}{theorem}
    
    \aliascntresetthe{corollary}

\theoremstyle{definition}
    \newaliascnt{definition}{theorem}
    
    \aliascntresetthe{definition}

 \theoremstyle{definition}
    \newaliascnt{remark}{theorem}
    \newtheorem{remark}[remark]{Remark}
    \aliascntresetthe{remark}

\theoremstyle{definition}
    \newaliascnt{proposition}{theorem}
    
    \aliascntresetthe{proposition}

\newcommand{\cc}{{\mathfrak{c}}}
\newcommand{\ca}{{\mathfrak{a}}}
\newcommand{\cs}{{\mathfrak{s}}}

\date{\today}

\begin{document}

\pagestyle{plain}
\title[Higher-Dimensional NLS Equation with Quasi-Periodic Initial Data]{The Weakly Nonlinear Schr\"odinger Equation in Higher Dimensions with Quasi-periodic Initial Data
}

\author{Fei Xu}
\address{\scriptsize (F. Xu)~Institute of Mathematics, Jilin University, Changchun 130012, P.R. China.}
\email{\color{magenta}stuxuf@outlook.com}
\thanks{F. X. is supported by the China post-doctoral grant (BX20240138). He would like to thank David Damanik (Rice University) and Yong Li (Jilin University) for their continuous guidance, Robert Schippa (University of California, Berkeley) for his comment on \cite{S2024}, Thomas Spencer (Institute for Advanced Study, Princeton) for sharing \cite{DSS20},  Jason Zhao (University of California, Berkeley) for the discussion on \cite{Z2024},
and the editor and anonymous referee for their constructive suggestions.
}

\begin{abstract}
In this paper, under the exponential/polynomial decay condition in Fourier space, we prove that the nonlinear solution to the quasi-periodic Cauchy problem for the weakly nonlinear Schr\"odinger equation in higher dimensions will asymptotically approach the associated linear solution within a specific time scale. The proof is based on a combinatorial analysis method present through diagrams. Our results and methods apply to {\em arbitrary} space dimensions and general power-law nonlinearities of the form $\pm|u|^{2p}u$, where $1\leq p\in\mathbb N$.
\end{abstract}

\maketitle

\tableofcontents

\section{Introduction}
Recently, the almost-periodic initial value problem for time-dependent PDEs has attracted significant attention. This interests originates from the so-called {\em Deift conjecture}: For the KdV equation, if the initial data is almost periodic, then the solution evolves almost periodically in time; see \cite{D}. This conjecture has been explored in various contexts: counterexamples \cite{D21,CKV}, quasi-periodic initial data \cite{DG}, almost-periodic initial data \cite{BDGL,LY,EVY}, Toda lattice (discrete version of the KdV equation) \cite{BDLV,ZC} and the generalized KdV equation \cite{DLX24JMPA}.

In the current paper, we study the higher-dimensional nonlinear Schr\"odinger equation (NLS for short)
\begin{align}\label{NLS}
{\rm i}\partial_tu+\Delta_du+\varepsilon|u|^{2p}u=0
\end{align}
with the quasi-periodic initial data
\begin{align}\label{ID}
u(0,\vec{x})=\sum_{\vec{n}\in\mathbb Z^\nu}c(\vec{n})e^{{\rm i}\langle\vec{n}\Omega,\vec{x}\rangle},\quad \vec{x}\in\mathbb R^d,
\end{align}
where $2\leq d\in\mathbb N$ represents the space dimension. Regarding these notations we will introduce them below. Specifically, we focus on the asymptotic behavior of the spatially quasi-periodic solutions with the same frequency matrix $\Omega$ as the initial data for the quasi-periodic Cauchy problem \eqref{NLS}-\eqref{ID}.

Below we introduce the notations used above.
\begin{itemize}
\item Throughout this paper, please keep in mind that all vectors are presented in row form. Furthermore we use $^{\top}$ to denote the transpose of a vector.
\item We use $t\in\mathbb R$ and $\vec{x}=(x_1,\cdots,x_d)\in\mathbb R^d$ to represent the time and spatial coordinates, respectively.
\item The solution $u(\cdot,\cdot):\mathbb R^{1+d}\rightarrow\mathbb C$ of \eqref{NLS} and the initial data $u(0,\cdot):\mathbb R^d\rightarrow\mathbb C$ are complex valued functions.
    \item $\partial_t$ stands for the time derivative, and $\Delta_d:=\sum_{j=1}^d\partial_{j}^2$ ($\partial_j$ denotes the partial derivative with respect to $x_j$-direction) denotes the Laplacian.
\item The strength parameter $p$ of nonlinearity is not less than $1$.
\item The sign of $\varepsilon$ sgn$(\varepsilon)\in\{+1,-1\}$ directly affects how waves interact with the medium: whether they tend to concentrate and form localized structure (focusing case: sgn$(\varepsilon)=+1$) or dispersive and spread out (defocusing case: sgn$(\varepsilon)=-1$); see \cite{KTV09JEMS}. Throughout this paper, we focus on the long-time asymptotic behavior of the quasi-periodic Cauchy problem \eqref{NLS}-\eqref{ID} within the time scale associated with the nonlinear parameter $\varepsilon$. Thus we assume that $0<|\varepsilon|\ll1$.
\item In the quasi-periodic Fourier space, we define the matrix as
\begin{align}
\Omega=\text{diag}\{\omega_1^{\top},\cdots,\omega_d^{\top}\}
=\begin{pmatrix}
\omega_1^{\top}&&\\
&\ddots&\\
&&\omega_d^{\top}
\end{pmatrix},
\end{align}
where $\omega_j=(\omega_j^1,\cdots,\omega_j^{\nu_j})\in\mathbb R^{\nu_j}$ is the frequency vector in the $x_j$-direction, and $2\leq\nu_j\in\mathbb N$ for all $j\in\{1,\cdots,d\}$.
\item Furthermore, we use $\vec{n}=(n_1,\cdots,n_d)\in\mathbb Z^{\nu_1}\times\cdots\times\mathbb Z^{\nu_d}\simeq\mathbb Z^\nu$ to denote the Pontryagin dual of the spatial variable $\vec{x}\in\mathbb R^d$, where $n_j=(n_{j,1},\cdots,n_{j,\nu_j})\in\mathbb Z^{\nu_j}$ for all $j\in\{1,\cdots,d\}$, and $\nu=\nu_1+\cdots+\nu_d$.
\item $\langle\vec{y},\vec{z}\rangle:=\vec{y}\cdot\vec{z}^{\top}=\sum_{j=1}^{N}y_jz_j$ denotes the Euclidean inner product, where $\vec{y}=(y_1,\cdots,y_N)\in\mathbb R^N, \vec{z}=(z_1,\cdots,z_N)\in\mathbb R^N$ and $2\leq N\in\mathbb N$.
\item Throughout this paper, we assume that the frequency matrix $\Omega$ satisfies certain non-resonant condition in the following sense: $\vec{n}\Omega=0\in\mathbb Z^{\nu}$ if and only if $\vec{n}=0\in\mathbb Z^\nu$. This is equivalent to $\omega_j$ being rationally independent for all $j \in\{1, \cdots, d\}$, meaning that $\langle n_j,\omega_j\rangle=0$ if and only if $n_j=0\in\mathbb Z^{\nu_j}$ for all $j\in\{1,\cdots,d\}$. For such a frequency matrix $\Omega$, the initial data $u(0,\cdot)$ is quasi-periodic in the $x_j$-direction with the frequency vector $\omega_j$ for all $j\in\{1,\cdots,d\}$.
\item $\{c(\vec{n})\}_{\vec{n}\in\mathbb Z^\nu}$ is a sequence of initial Fourier data (it satisfies some decay condition in Fourier space which will be introduced below).
\end{itemize}

%

This paper examines quasi-periodic functions in higher dimensions as initial data for the NLS equation \eqref{NLS}. Specifically, we focus on functions of the form \eqref{ID} in the spatial variables. This means that $u(0,\cdot)$ is quasi-periodic in each direction $x_j$ with frequency vector $\omega_j\in\mathbb R^{\nu_j}$ provided that it is rationally independent.

It should be emphasized that these functions differ from periodic functions and those with the vanishing property at infinity, as they exhibit oscillatory behavior and cannot be viewed as functions on the circle; see \cite{Stein}.

For the spatially quasi-periodic function in higher dimensions, we have the following orthogonal property.
\begin{lemma}[Orthogonality of $\{e^{{\rm i}\langle\vec{n}\Omega,\vec{x}\rangle}\}_{\vec{n}\in\mathbb Z^\nu}$]\label{orglemm}
Let $\Omega=\text{diag}\{\omega_1^{\top},\cdots,\omega_d^{\top}\}$ be rationally independent frequency matrix (i.e., $\vec{n}\Omega=0\Rightarrow\vec{n}=0\in\mathbb Z^\nu$, or rather, $\omega_jn_j=0\Rightarrow n_j=0\in\mathbb Z^{\nu_j}$ for all $j\in\{1,\cdots,d\}$). For any $\vec{n}=(n_1,\cdots,n_d)\in\mathbb Z^{\nu_1+\cdots+\nu_d}$ and $\vec{m}=(m_1,\cdots,m_d)\in\mathbb Z^{\nu_1+\cdots+\nu_d}$, we have
\begin{align}\label{org}
\lim_{\substack{L_j\rightarrow+\infty\\ j=1,\cdots,d}}\frac{1}{\prod_{j=1}^{d}2L_j}\int_{-L_1}^{L_1}\cdots\int_{-L_d}^{L_d}e^{{\rm i}\langle(\vec{n}-\vec{m})\Omega,\vec{x}\rangle}{\rm d}\vec{x}=\prod_{j=1}^{d}\delta_{n_jm_j},
\end{align}
where $\delta$ stands for the Kronecker delta defined by letting $\delta_{n_jm_j}=1$ if $n_j=m_j$ and $\delta_{n_jm_j}=0$ if $n_j\neq m_j$.
\end{lemma}

The NLS equation \eqref{NLS} is one of the most important mathematical physics equations. It is a mathematical model with applications across various fields of physics. Its relevance to Bose-Einstein condensation highlights how collective behavior in quantum gases can be described through mean-field theories. In nonlinear optics, the Kerr effect, in which the refractive index of a material changes with light intensity, can be modeled using the NLS equation, facilitating the study of pulse propagation in optical fibers. Additionally, the phenomenon of the rogue waves, which are unexpectedly large and powerful ocean waves, can also be analyzed using this framework, as the NLS equation captures the envelope dynamics of these waves; see \cite{I08AM,KTV09JEMS,G2000,dgv,GGKS} and the references therein.

From the perspective of mathematics, the NLS equation \eqref{NLS} enjoys the following scaling symmetry: $u_\lambda(t,x):=\lambda^{\theta}u(\lambda^2t,\lambda \vec{x})$ is a solution to \eqref{NLS} if and only if $\theta=\frac{2}{2p}$, where $\lambda>0$; see \cite{D23IMRN}. This scaling property is fundamental as it reveals the self-similar nature of solutions to the NLS equation, providing insight into how solutions behave under rescaling of time and space.

Furthermore, consider the $L^2$-Sobolev space $\dot{H}^{s}(\mathbb R^d)$, with norm defined as $$\|f\|_{\dot{H}^{s}(\mathbb R^d)}:=\langle(-\Delta_d)^{s/2}f,(-\Delta_d)^{s/2}f\rangle_{\mathcal L^2(\mathbb R^d)}^{1/2}$$ for any  $f\in\dot{H}^{s}(\mathbb R^d)$, we know that $$\|u(t,\cdot)\|_{\dot{H}^{s}(\mathbb R^d)}=\|u_\lambda(t,\cdot)\|_{\dot{H}^{s}(\mathbb R^d)}$$ if the Sobolev-critical index $s_c=\frac{d}{2}-\frac{2}{2p}$; see \cite{D23IMRN}. This relationship illustrates how the norms of solutions are preserved under the scaling transformation, which is critical for understanding the long-time behavior of solutions and their stability.

In particular, there are two remarkable cases regarding the behavior of solutions to the NLS equation. The first one is the so-called {\em mass-critical (pseudoconformal)} case, where the Sobolev-critical index $s_c=0$. In this case, we have the specific condition $2p=\frac{4}{d}$. This is due to the fact that the solution $u$ of \eqref{NLS} preserves the mass ($L^2$-norm of a solution)
\[M[u](t):=\int_{\mathbb R^d}|u(t,\vec{x})|^2{\rm d}\vec{x}=M[u](0).\]
The second one is the so-called {\em energy-critical} case, where the Sobolev-critical index $s_c=1$, i.e., $2p=\frac{4}{d-2}$. This is due to the fact that the solution $u$ of \eqref{NLS} preserves the (Hamiltonian) energy
\[
H[u](t):=\frac{1}{2}\int|\nabla u(t,\vec{x})|^2{\rm d}\vec{x}+\frac{\varepsilon}{2p+1}\int|u(t,\vec{x})|^{2p+1}{\rm d}\vec{x}=H[u](0).
\]
These indicate a balance between the nonlinearity and dispersion effects in this equation, leading to rich dynamical behavior; see \cite{CW89,KTV09JEMS,D23IMRN} and the references therein.
%
%
%
%

The well-posedness analysis of the Cauchy problem for the higher-dimensional NLS equation has been extensively studied; see \cite{CW89,CW90,B99JAMS,G2000,T05,RV07AJM,I08AM,KTV09JEMS,DLM19,D23IMRN} and the references therein.

%
%

Based on these fruitful results, this paper focuses on the asymptotic dynamics of spatially quasi-periodic solutions with the same frequency matrix as the initial data \eqref{ID} for the quasi-periodic Cauchy problem \eqref{NLS}-\eqref{ID}. That is, such a solution is given by the following Fourier expansion:
\begin{align}\label{S}
u(t,\vec{x})=\sum_{\vec{n}\in\mathbb Z^\nu}c(t,\vec{n})e^{{\rm i}\langle\vec{n}\Omega,\vec{x}\rangle},
\end{align}
where $\{c(t,\vec{n})\}_{\vec{n}\in\mathbb Z^\nu}$ is a sequence of unknown time-dependent Fourier coefficients. We are interested in the following question:
\begin{center}
{\bf Q}:~\em What is the behavior of the spatially quasi-periodic solution \eqref{S} \\to the quasi-periodic Cauchy problem \eqref{NLS}-\eqref{ID} \\as time tends to infinity?
\end{center}

To the best of our knowledge, the asymptotic behavior of the quasi-periodic Cauchy problem for the weakly nonlinear Schr\"odinger equation in higher dimensions has not been studied in previous works.

To answer this question, we first determine the time scale for existence and uniqueness, and then study the asymptotic behavior within this specific scale. Regarding the study of the existence and uniqueness problem (without smallness assumption on the nonlinearity), in the one-dimensional case, Damanik-Li-Xu \cite{F2024CULA,DLX1} and Papenburg \cite{P2024} obtained the local existence and uniqueness results at the same time.  The result of \cite{DLX1} is based on the assumption of polynomially decaying Fourier coefficients. Later, Schippa \cite{S2024} studied the cubic nonlinear Sch\"odinger equation in two dimensions with polynomially decaying Fourier data and proved that the regularity is sharp. In higher dimensions, Zhao \cite{Z2024} and Xu \cite{Xu2024} obtained the local existence and uniqueness result almost simultaneously\footnote{As this paper was being completed, Zhao submitted an interesting paper to arXiv in which he studied the dispersive equations with bounded  data in higher dimensions; see \cite{Z2024}.}. The former also studied the continuous dependence problem, and the latter obtained the asymptotic behaviour within the time scale.

Additionally, Oh investigated the well-posedness for the NLS equation in one dimension with almost periodic initial data and limit-periodic initial data; see \cite{O15l,O15} respectively. We also refer the reader to \cite{DSS20, K2023} on this topic.


Our results and methods apply to  $(d,p,\varepsilon)\in[1,\infty)\times[1,\infty)\times\{\pm1\}$. For the sake of convenience, we consider only the simplest but nontrivial case:  $(d,p,\varepsilon)=(2,1,+1)$.
That is, consider the following cubic nonlinear Schr\"odinger equation
\begin{align}\label{NLS2}
{\rm i}\partial_tu+(\partial_x^2+\partial_y^2)u+|u|^2u=0
\end{align}
in  two dimensions with the quasi-periodic initial data
\begin{align}\label{ID2}
u(0,x,y)=
\sum_{(m,n)\in\mathbb Z^{\nu_1}\times\mathbb Z^{\nu_2}}c(m,n)
e^{{\rm i}\big(\langle m,\omega\rangle x+\langle n,\omega^\prime\rangle y\big)}, \quad (x,y)\in\mathbb R^2.
\end{align}

The main results of this paper are stated as follows.

\begin{theorem}\label{thm}
Consider the quasi-periodic Cauchy problem \eqref{NLS2}-\eqref{ID2}, and make the following assumptions:
\begin{itemize}
  \item The frequency vectors $\omega$ and $\omega^\prime$ are rationally independent, that is,
$\langle m,\omega\rangle =0, m\in\mathbb Z^{\nu_1}\Leftrightarrow m=0\in\mathbb Z^{\nu_1}$; $\langle n,\omega^\prime\rangle =0, n\in\mathbb Z^{\nu_2}\Leftrightarrow n=0\in\mathbb Z^{\nu_2}$.
  \item The initial Fourier data satisfies the $(\kappa_1,\kappa_2)$-exponential decaying condition, i.e.,
\begin{align}\label{ic}
|c(m,n)|\leq e^{-(\kappa_1|m|+\kappa_2|n|)}, \quad (m,n)\in\mathbb Z^{\nu_1}\times\mathbb Z^{\nu_2},
\end{align}
where $0<\kappa_j\leq1$ for $j\in\{1,2\}$, and $|\cdot|$ denotes the $\ell^1$-norm of a vector, that is, $|m|=\|m\|_{\ell^1}=\sum_{j=1}^{\nu_1}|m_j|$ for $m=(m_j)_{1\leq j\leq\nu_1}\in\mathbb Z^{\nu_1}$;  $|n|=\|n\|_{\ell^1}=\sum_{j=1}^{\nu_1}|n_j|$ for $n=(n_j)_{1\leq j\leq\nu_1}\in\mathbb Z^{\nu_2}$.
\end{itemize}

There exists a positive constant (see \eqref{te}) $$T_\varepsilon=\frac{4}{27}\left(\frac{\kappa_1}{6}\right)^{\nu_1}\left(\frac{\kappa_2}{6}\right)^{\nu_2}
\varepsilon^{-1}$$ and a sequence of uniform-in-time $(\kappa_1/2,\kappa_2/2)$-exponential decaying Fourier coefficients, that is,  $$|\cc(t,m,n)|\lesssim e^{-\left(\frac{\kappa_1}{2}|m|+\frac{\kappa_2}{2}|n|\right)},\quad(t,m,n)\in[0,T_\varepsilon]\times\mathbb Z^{\nu_1}\times\mathbb Z^{\nu_2},$$
such that 
$$\mathfrak u(t,x,y)=\sum_{(m,n)\in\mathbb Z^{\nu_1}\times\mathbb Z^{\nu_2}}\mathfrak c(t,m,n)e^{{\rm i}(\langle m,\omega\rangle x+\langle n,\omega^\prime\rangle y)},\quad (t,x,y)\in[0,T_\varepsilon]\times\mathbb R\times\mathbb R,
$$
is the unique spatially quasi-periodic solution to  
\eqref{NLS2}-\eqref{ID2}.

Set
$$\mathfrak u_0(t,x,y)=\sum_{(m,n)\in\mathbb Z^{\nu_1}\times\mathbb Z^{\nu_2}}e^{-{\rm i}\big(\langle m,\omega\rangle^2+\langle n,\omega^\prime\rangle^2\big)t}c(m,n)e^{{\rm i}(\langle m,\omega\rangle x+\langle n,\omega^\prime\rangle y)}.$$
Let $t\sim|\varepsilon|^{-1+\eta}$ with $0<\eta\ll1$. Then the nonlinear solution $\mathfrak u$ will asymptotically approach the linear solution $\mathfrak u_0$ in the sense of
\begin{align}
\|(\mathfrak u-\mathfrak u_0)(t,x,y)\|_{\mathcal H^{(\rho_1,\rho_2)}(\mathbb R^2)}\rightarrow 0, \quad\text{as}~\varepsilon\rightarrow0,
\end{align}
where $\|f\|_{\mathcal H^{(\rho_1,\rho_2)}(\mathbb R^2)}=\|e^{\rho_1|m|+\rho_2|n|}\hat f(m,n)\|_{\ell^2_{(m,n)}(\mathbb Z^{\nu_1}\times\mathbb Z^{\nu_2})}$ with $0<\frac{\kappa_j}{2}-2\rho_j\leq1$ for $j\in\{1,2\}$.
\end{theorem}
\begin{remark}
(a)~\autoref{thm} applies to $(d,p,\varepsilon)\in[1,\infty)\times[1,\infty)\times\{\pm1\}$. This includes the following two important cases:
\begin{itemize}
  \item Mass-critical case: $2p=4/d$, e.g., \cite[$(1,2)$]{F2024CULA,P2024,DLX1}, \cite[$(2,1)$]{KTV09JEMS,S2024}; see \autoref{m1}.
  \item Energy-critical case: $2p=4/(d-2)$, e.g., \cite[$(3,2)$]{B99JAMS,G2000}, \cite[$(4,1)$]{RV07AJM,DLM19}; see \autoref{m2}.
\end{itemize}
\begin{figure}[H]
\centering
\begin{minipage}{0.45\textwidth}
\centering
\begin{tikzpicture}[scale=0.5,transform shape]
\node at (1,1) {\color{blue}$\bullet$};
\node at (1,2) {\color{red}$\bullet$};
\node at (1,3) {$\bullet$};
\node at (1,4) {$\bullet$};
\node at (1,5) {$\bullet$};
\node at (1,6) {$\bullet$};

\node at (2,1) {\color{red}$\bullet$};
\node at (2,2) {$\bullet$};
\node at (2,3) {$\bullet$};
\node at (2,4) {$\bullet$};
\node at (2,5) {$\bullet$};
\node at (2,6) {$\bullet$};

\node at (3,1) {$\bullet$};
\node at (3,2) {$\bullet$};
\node at (3,3) {$\bullet$};
\node at (3,4) {$\bullet$};
\node at (3,5) {$\bullet$};
\node at (3,6) {$\bullet$};

\node at (4,1) {$\bullet$};
\node at (4,2) {$\bullet$};
\node at (4,3) {$\bullet$};
\node at (4,4) {$\bullet$};
\node at (4,5) {$\bullet$};
\node at (4,6) {$\bullet$};

\node at (5,1) {$\bullet$};
\node at (5,2) {$\bullet$};
\node at (5,3) {$\bullet$};
\node at (5,4) {$\bullet$};
\node at (5,5) {$\bullet$};
\node at (5,6) {$\bullet$};

\node at (6,1) {$\bullet$};
\node at (6,2) {$\bullet$};
\node at (6,3) {$\bullet$};
\node at (6,4) {$\bullet$};
\node at (6,5) {$\bullet$};
\node at (6,6) {$\bullet$};
    \draw[->] (0,0) -- (6.5,0) node[right] {$d$};
    \draw[->] (0,0) -- (0,6.5) node[above] {$p$};

    \foreach \x in {1,2,3,4,5,6} 
        \draw (\x,1pt) -- (\x,-3pt) node[below] {\x}; 

    \foreach \y in {1,2,3,4,5,6} 
        \draw (1pt,\y) -- (-3pt,\y) node[left] {\y}; 

    \draw[black, dashed, domain=0.35:6.35, samples=100] plot (\x, {2/\x}) node[right] {$2p= \frac{4}{d}$};
\end{tikzpicture}
\captionof{figure}{\small Mass-critical case}
\label{m1}
\end{minipage}
\begin{minipage}{0.45\textwidth}
\centering
\begin{tikzpicture}[scale=0.5,transform shape]
\node at (3,1) {\color{blue}$\bullet$};
\node at (3,2) {\color{red}$\bullet$};
\node at (3,3) {$\bullet$};
\node at (3,4) {$\bullet$};
\node at (3,5) {$\bullet$};
\node at (3,6) {$\bullet$};

\node at (4,1) {\color{red}$\bullet$};
\node at (4,2) {$\bullet$};
\node at (4,3) {$\bullet$};
\node at (4,4) {$\bullet$};
\node at (4,5) {$\bullet$};
\node at (4,6) {$\bullet$};

\node at (5,1) {$\bullet$};
\node at (5,2) {$\bullet$};
\node at (5,3) {$\bullet$};
\node at (5,4) {$\bullet$};
\node at (5,5) {$\bullet$};
\node at (5,6) {$\bullet$};

\node at (6,1) {$\bullet$};
\node at (6,2) {$\bullet$};
\node at (6,3) {$\bullet$};
\node at (6,4) {$\bullet$};
\node at (6,5) {$\bullet$};
\node at (6,6) {$\bullet$};

\node at (7,1) {$\bullet$};
\node at (7,2) {$\bullet$};
\node at (7,3) {$\bullet$};
\node at (7,4) {$\bullet$};
\node at (7,5) {$\bullet$};
\node at (7,6) {$\bullet$};
    \draw[->] (0,0) -- (7.5,0) node[right] {$d$}; 
    \draw[->] (0,0) -- (0,6.5) node[above] {$p$}; 

    \foreach \x in {2,3,4,5,6,7} 
        \draw (\x,1pt) -- (\x,-3pt) node[below] {\x}; 

    \foreach \y in {1,2,3,4,5,6} 
        \draw (1pt,\y) -- (-3pt,\y) node[left] {\y}; 

    \draw[black, dashed, domain=2.35:7.5, samples=100] plot (\x, {2/(\x-2)}) node[right] {$2p= \frac{4}{d-2}$};

\end{tikzpicture}
\captionof{figure}{\small Energy-critical case}
\label{m2}
\end{minipage}
\end{figure}

(b)~\autoref{thm} applies to the case of polynomially decaying initial Fourier data, that is, \eqref{ic} is replaced by
\begin{align*}
|c(m,n)|\leq(1+|m|)^{-r_1}(1+|n|)^{-r_2},\quad(m,n)\in\mathbb Z^{\nu_1}\times\mathbb Z^{\nu_2},
\end{align*}
for sufficient large $r_1$ and $r_2$; see \cite{S2024}, in which he proved the sharp polynomial regularity for the cubic NLS equation in two spatial dimensions.

(c)~Bourgain studied the time quasi-periodic solutions to Hamiltonian perturbations of the two-dimensional linear Schr\"odinger equation in a periodic setting; see \cite{b98}. In contrast, we investigated the asymptotic behavior of higher-dimensional weakly nonlinear Schr\"odinger equations in a quasi-periodic setting, without time periodicity.
\end{remark}


The proof of \autoref{thm} follows the subsequent process:
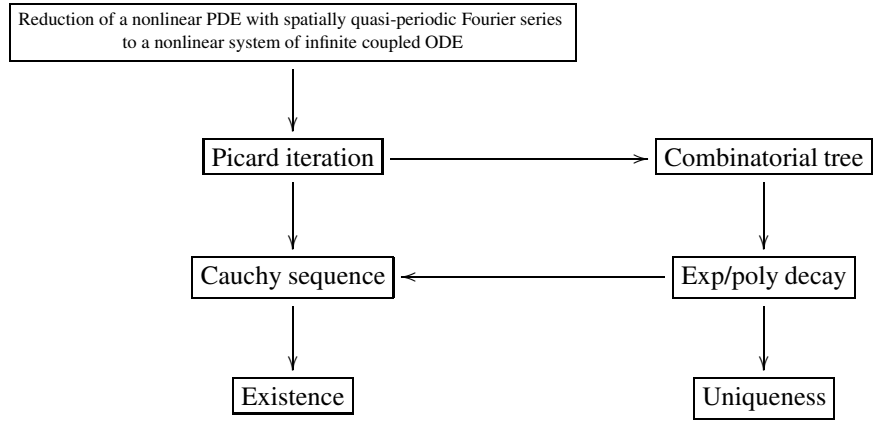
\[
\xymatrix{
\boxed{\substack{\text{Reduction of a nonlinear PDE with spatially quasi-periodic Fourier series }\\[0.5mm]\text{to a nonlinear system of infinite coupled ODE}}}
\ar[d]\\
\boxed{\text{Picard iteration}}
\ar[d]
\ar[r]&
\boxed{\text{Combinatorial tree}}\ar[d]
\\
\boxed{\text{Cauchy sequence}}\ar[d]&
\boxed{\text{Exp/poly decay}}
\ar[l]\ar[d]
\\
\boxed{\text{Existence}}&\boxed{\text{Uniqueness}}
}
\]

The proof employs an explicit combinatorial analysis method to examine the Picard iteration, a technique successfully utilized in one-dimensional problems; see \cite{C2007,DG,DLX24JMPA,DLX1,DLX2}. In this paper, we extend this method to address the higher-dimensional problem.

In the higher-dimensional case, the primary challenge lies in analyzing the combinatorial structure governing the growth of the Picard iteration, which is influenced by the increase in spatial dimensions. This contrasts with the one-dimensional case, where such complexity is less pronounced. To address this difficulty, we introduce a diagram representation method to facilitate the analysis of the Picard iteration. The key idea is to separate the elements in the domain based on their spatial direction.

To be more precise, each diagram corresponds to the nonlinear part of an element in the Picard sequence. In each diagram, we identify the domain of the multi-linear operator in the context of the Picard iteration as follows: we first gather all the dual lattice points associated with the same spatial direction into a cohesive object, then pull them back to the original domain. These two processes are mutually reversible. This split-combine method clarifies the combinatorial structure of the Picard iteration, making the summation function and necessary estimates feasible. We refer the reader to \autoref{fig0}, \autoref{fig1}, \autoref{fig3}, and \autoref{fig4} for further details. Additionally, in the proof of uniqueness, we utilize such a diagram to re-label the lattice points; see \autoref{fig5}, \autoref{fig87}, and \autoref{fig7} for specifics.

%

%
%

\vspace{2mm}

{\bf Outline Of The Paper}~In \autoref{seccub}, we examine the cubic NLS equation in two dimensions. This section will be divided into \autoref{ode}, \autoref{sectree}, and \autoref{secdecay}, where we will derive the nonlinear system of infinite coupled ODEs that arises from the nonlinear NLS equation, along with the spatially quasi-periodic Fourier series. We will also establish the combinatorial tree structure for the Picard sequence and demonstrate that the Picard sequence exhibits uniform-in-time exponential decay. In \autoref{cau}, we will show that the Picard sequence is a Cauchy sequence. Finally, in \autoref{secasy}, we will prove the main asymptotic result.

\section{The Cubic NLS Equation in Two Dimensions: $(d,p)=(2,1)$}\label{seccub}
In this section we focuses on the simplest case: $(d,p)=(2,1)$, which corresponds to \eqref{NLS2}-\eqref{ID2}.
\subsection{The Nonlinear System of Infinite Coupled ODEs}\label{ode}
The spatially quasi-periodic solutions with the same frequency matrix diag$\{\omega^\top,{\omega^\prime}^{\top}\}$ for the above Cauchy problem have the following Fourier expansion
\begin{align}\label{S2}
u(t,x,y)=
\sum_{(m,n)\in\mathbb Z^{\nu_1}\times\mathbb Z^{\nu_2}}c(t,m,n)
e^{{\rm i}\big(\langle m,\omega\rangle x+\langle n,\omega^\prime\rangle y\big)}, \quad (x,y)\in\mathbb R^2.
\end{align}

By Lemma \ref{orglemm}, substituting \eqref{S2} into \eqref{NLS2} yields the following nonlinear system of infinite coupled ODEs
\begin{align*}
{\rm i}\partial_tc(t,m,n)&-\left(\langle m,\omega\rangle^2+\langle n,\omega^\prime\rangle^2\right)c(t,m,n)\nonumber\\
&+\varepsilon
\sum_{\substack{(m,n)=\sum_{j=1}^3(-1)^{j-1}(m_j,n_j)\\(m_j,n_j)\in\mathbb Z^{\nu_1}\times\mathbb Z^{\nu_2}\\j=1,2,3}}\prod_{j=1}^{3}\left\{c(t,m_j,n_j)\right\}^{\ast^{[j-1]}}=0,
\end{align*}
where, for any complex number $z\in\mathbb C$,
\begin{align*}
z^{\ast^{[m]}}=
\begin{cases}
z,& m\in2\mathbb N;\\
\bar{z}, & m\in2\mathbb N+1.
\end{cases}
\end{align*}
Some operational properties of $^{\ast^{[\cdot]}}$ can be found in \cite[Subsection 2.2]{DLX1}.

According to the initial data \eqref{ID2}, we have $c(0,m,n)=c(m,n)$ for all $(m,n)\in\mathbb Z^{\nu_1}\times\mathbb Z^{\nu_2}$.

It follows from the Duhamel's principle that
\begin{align*}
c(t,m,n)&=\Phi^{t}(m,n)c(m,n)\\
&+{\rm i}\varepsilon\int_0^t\Phi^{t-s}(m,n)
\sum_{\substack{(m,n)=\sum_{j=1}^3(-1)^{j-1}(m_j,n_j)\\(m_j,n_j)\in\mathbb Z^{\nu_1}\times\mathbb Z^{\nu_2}\\j=1,2,3}}\prod_{j=1}^{3}\left\{c(s,m_j,n_j)\right\}^{\ast^{[j-1]}}{\rm d}s,
\end{align*}
where $\Phi^t(m,n)=e^{-{\rm i}\big(\langle m,\omega\rangle^2+\langle n,\omega^\prime\rangle^2\big)t}$.

Define the Picard sequence $\{c_k(t,m,n)\}_{k\geq0}$ as follows: start with the linear solution as the initial guess, that is, $c_0(t,m,n)=\Phi^t(m,n)c(m,n)$. For $k\geq1$,
\begin{align}
&c_k(t,m,n)\nonumber\\
=~&c_0(t,m,n)\nonumber\\
&+{\rm i}\varepsilon\int_0^t\Phi^{t-s}(m,n)
\sum_{\substack{(m,n)=\sum_{j=1}^3(-1)^{j-1}(m_j,n_j)\\(m_j,n_j)\in\mathbb Z^{\nu_1}\times\mathbb Z^{\nu_2}\\j=1,2,3}}\prod_{j=1}^{3}\left\{c_{k-1}(s,m_j,n_j)\right\}^{\ast^{[j-1]}}{\rm d}s.\label{pi}
\end{align}

\subsection{Combinatorial Tree of the Picard Iteration}\label{sectree}
In this subsection, we will prove that the Picard sequence takes the form of a combinatorial tree, as detailed in \autoref{treethm}.
\begin{theorem}\label{treethm}
For all $k\geq1$, we have
\begin{align}\label{ctree}
c_k(t,m,n)=\sum_{\gamma^{(k)}\in\Gamma^{(k)}}
\sum_{\substack{\spadesuit^{(k)}\in{\mathrm D}^{(k,\gamma^{(k)})}\\\lambda(\spadesuit^{(k)})=(m,n)}}
\mathfrak C^{(k,\gamma^{(k)})}(\spadesuit^{(k)})
\mathfrak I^{(k,\gamma^{(k)})}(t,\spadesuit^{(k)})
\mathfrak F^{(k,\gamma^{(k)})}(\spadesuit^{(k)}).
\end{align}
\end{theorem}
\begin{proof}
We divide the proof into the following five steps.

{\bf Step 1. Label each term in the Picard iteration.}

Define the branch set $\{\Gamma^{(k)}\}_{k\geq1}$ is defined as follows:
\begin{align}
\Gamma^{(k)}=
\begin{cases}
\{0,1\},&k=1;\\
\{0\}\cup\Gamma^{(k-1)}\times\Gamma^{(k-1)}\times\Gamma^{(k-1)},&k\geq2.
\end{cases}
\end{align}

{\bf Step 2. Analysis of the Linear Part.}

Notice that the linear part $c_0(t,m,n)$ can be labeled by $0\in\Gamma^{(k)}$, where $k\geq1$. To be precise,
\begin{align}\label{c0l}
c_0(t,m,m)=\sum_{\gamma^{(k)}=0\in\Gamma^{(k)}}
\sum_{\substack{\lambda(\spadesuit^{(k)})=(m,n)\\\spadesuit^{(k)}\in{\mathrm D}^{(k,\gamma^{(k)})}}}
\mathfrak C^{(k,\gamma^{(k)})}(\spadesuit^{(k)})\mathfrak I^{(k,\gamma^{(k)})}(t,\spadesuit^{(k)})\mathfrak F^{(k,\gamma^{(k)})}(\spadesuit^{(k)}).
\end{align}
where the definitions of $\mathrm D,\lambda,\mathfrak C,\mathfrak I$ and $\mathfrak F$ are provided in the following \autoref{fig0}.
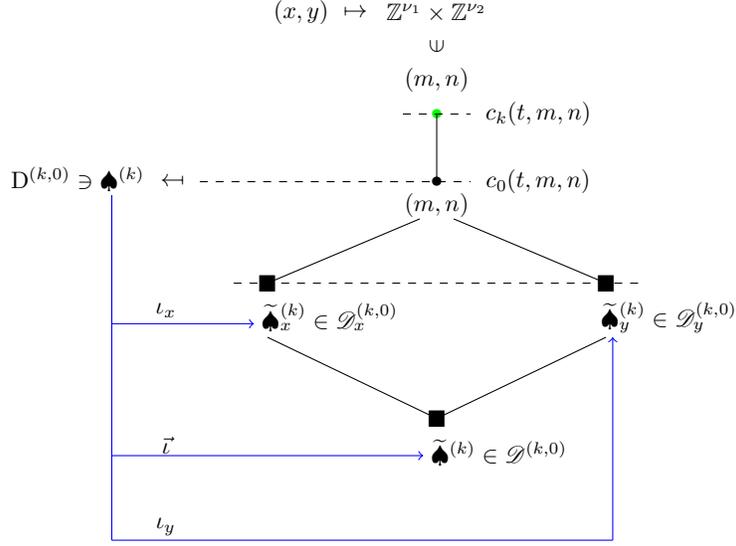
\begin{figure}[htpb!]
\centering
\begin{tikzpicture}[scale=0.9,transform shape]
\node at (-2,2.5) {$(x,y)$};
\node at (-1.2,2.5) {$\mapsto$};
\node at (0,2.5) {$\mathbb Z^{\nu_1}\times\mathbb Z^{\nu_2}$};
\node at (0,2) {\rotatebox{90}{$\in$}};
\node at (0,1.5) {$(m,n)$};
\node at (0,1) {\color{green}$\bullet$};
\draw [dashed] (-0.5,1) -- (0.5,1);
\node at (1.5,1) {$c_k(t,m,n)$};
\draw (0,0) -- (0,1);
\node at (0,0) {$\bullet$};
\draw [dashed] (-3.5,0) -- (0.5,0);
\node at (0,-0.35) {$(m,n)$};
\node at (1.5,0) {$c_0(t,m,n)$};

\node at (-5.3,0.04) {$\mathrm D^{(k,0)}\ni\spadesuit^{(k)}$};
\node at (-3.9,0.04) {\rotatebox{180}{$\mapsto$}};

\coordinate (A) at (-4.8,-0.2);
\coordinate (B) at (-4.8,-2.1);
\coordinate (C) at (-2.7,-2.1);

\coordinate (D) at (-4.8,-4.05);
\coordinate (E) at (-0.2,-4.05);

\coordinate (F) at (-4.8,-5.3);
\coordinate (G) at (2.6,-5.3);
\coordinate (H) at (2.6,-2.3);

\draw [blue] (A)--(B);
\draw [->,blue] (B)--(C);
\draw [blue] (A)--(D);
\draw [->,blue] (D)--(E);
\draw [blue] (A)--(F);
\draw [blue] (F)--(G);
\draw [->,blue] (G)--(H);

\node at (-2.5,-1.5) {$\blacksquare$};
\node at (2.5,-1.5) {$\blacksquare$};

\draw (-0.25,-0.55)--(-2.5,-1.5);
\draw (0.25,-0.55)--(2.5,-1.5);
\draw [dashed](-3,-1.5)--(3,-1.5);

\node at (-1.58,-2) {$\widetilde\spadesuit_x^{(k)}\in\mathscr D_x^{(k,0)}$};
\node at (3.43,-2) {$\widetilde\spadesuit_y^{(k)}\in\mathscr D_y^{(k,0)}$};

\node at (-4,-1.9) {$\iota_x$};
\node at (-4,-3.9) {$\vec{\iota}$};
\node at (-4,-5.1) {$\iota_y$};

\node at (0,-3.5) {$\blacksquare$};
\draw (-2.5,-2.3)--(0,-3.5);
\draw (2.5,-2.3)--(0,-3.5);

\node at (0.92,-4) {$\widetilde\spadesuit^{(k)}\in\mathscr D^{(k,0)}$};
\end{tikzpicture}
\caption{Analysis of the linear part $c_0(t,m,n)$. The green circle ${\color{green}\bullet}$ in the first dashed line represents the dual variables $(m,n)$ in $c_k(t,m,n)$, the black circle $\bullet$ in the second dashed line represents the variables in the linear part $c_0(t,m,n)$.
The black circle $\bullet$ and the green circle ${\color{green}\bullet}$ denote elements associated with $\mathbb Z^{\nu_1}\times\mathbb Z^{\nu_2}$. The black square $\blacksquare$ denote elements associated with $\mathbb Z^{\nu_1}$ or $\mathbb Z^{\nu_2}$. The blue lines represent the mappings $\vec{\iota},\iota_x$ and $\iota_y$.}
\label{fig0}
\end{figure}

Firstly, on the branch $0\in\Gamma^{(k)}$, the domain $\mathrm D$ is defined as follows:
\begin{align*}
{\mathrm D}^{(k,0)}=\mathbb Z^{\nu_1}\times\mathbb Z^{\nu_2}
\stackrel{\vec{\iota}}{\simeq}
\mathscr D_x^{(k,0)}\times\mathscr D_y^{(k,0)}=\mathscr D^{(k,0)},
\end{align*}
where
$\mathscr D_x^{(k,0)}=\mathbb Z^{\nu_1},
\mathscr D_y^{(k,0)}=\mathbb Z^{\nu_2}$,
and
\begin{align*}
(\iota_x,\iota_y)=\vec{\iota}:\mathrm D^{(k,0)}&\rightarrow\mathscr D^{(k,0)}\\
\mathbb Z^{\nu_1}\times\mathbb Z^{\nu_2}\ni(m,n)=\spadesuit^{(k)}&\mapsto\widetilde\spadesuit^{(k)}=\vec{\iota}(\spadesuit^{(k)})\\
&\hspace{13.5mm}=\left(\iota_x(\spadesuit^{(k)}),\iota_y(\spadesuit^{(k)})\right)\\
&\hspace{13.5mm}=(\widetilde\spadesuit_x^{(k)},\widetilde\spadesuit_y^{(k)})\\
&\hspace{13.5mm}=(m,n).
\end{align*}

Secondly, on the branch $0\in\Gamma^{(k)}$, the summation function $\lambda$ is defined as follows:
\begin{align*}
\lambda(\spadesuit^{(k)})&=(\vec{\cc\ca\cs}\circ\vec{\iota})(\spadesuit^{(k)})\\
&=\left((\cc\ca\cs_x\circ\iota_x)(\spadesuit^{(k)}),(\cc\ca\cs_y\circ\iota_y)(\spadesuit^{(k)})\right)\\
&=\left(\cc\ca\cs_x(\widetilde\spadesuit_x^{(k)}),\cc\ca\cs_y(\widetilde\spadesuit_y^{(k)})\right)\\
&=(m,n),\nonumber
\end{align*}
where $\vec{\cc\ca\cs}=(\cc\ca\cs_x,\cc\ca\cs_y)$, and
\begin{align*}
\cc\ca\cs_x: \cup_{s=1}^\infty(\mathbb Z^{\nu_1})^s&\rightarrow\mathbb Z^{\nu_1},\\
(\mathbb Z^{\nu_1})^s\ni(q_1,\cdots,q_s)&\mapsto\sum_{j=1}^{s}(-1)^{j-1}q_j;\\
\cc\ca\cs_y: \cup_{s=1}^\infty(\mathbb Z^{\nu_2})^s&\rightarrow\mathbb Z^{\nu_2},\\
 (\mathbb Z^{\nu_2})^s\ni(q_1,\cdots,q_s)&\mapsto\sum_{j=1}^{s}(-1)^{j-1}q_j.
\end{align*}

Finally,
\begin{align*}
\mathfrak C^{(k,0)}(\spadesuit^{(k)})&=c(\lambda(\spadesuit^{(k)}));\\
\mathfrak I^{(k,0)}(t,\spadesuit^{(k)})&=\Phi^{t}(\lambda(\spadesuit^{(k)}));\\
\mathfrak F^{(k,0)}(\spadesuit^{(k)})&=1.
\end{align*}

{\bf Step 3. Analysis of the First Iteration}~

For the definitions of $\mathrm D^{(k,\gamma^{(k)})}, \lambda, \mathfrak C,\mathfrak I$ and $\mathfrak F$ on the branch $0\in\Gamma^{(k)}$, we refer to {\bf Step 2}.

For the nonlinear part $(c_1-c_0)(t,m,n)$, it follows from \eqref{pi} that
\begin{align*}
&(c_1-c_0)(t,m,n)\\
=~&{\rm i}\varepsilon\int_0^t\Phi^{t-s}(m,n)
\sum_{\substack{(m,n)=\sum_{j=1}^3(-1)^{j-1}(m_j,n_j)\\(m_j,n_j)\in\mathbb Z^{\nu_1}\times\mathbb Z^{\nu_2}\\j=1,2,3}}\prod_{j=1}^{3}\left\{c_{0}(s,m_j,n_j)\right\}^{\ast^{[j-1]}}{\rm d}s\\
=~&\sum_{\substack{(m,n)=\sum_{j=1}^3(-1)^{j-1}(m_j,n_j)\\(m_j,n_j)\in\mathbb Z^{\nu_1}\times\mathbb Z^{\nu_2}\\j=1,2,3}}\prod_{j=1}^{3}\left\{c(m_j,n_j)\right\}^{\ast^{[j-1]}}\\
~&\times
\int_0^t\Phi^{t-s}(m,n)\prod_{j=1}^{3}\left\{\Phi^s(m_j,n_j)\right\}^{\ast^{[j-1]}}{\rm d}s\times({\rm i}\varepsilon).
\end{align*}

We use $1\in\Gamma^{(1)}$ to label the nonlinear part $(c_1-c_0)(t,m,n)$ of the first iteration $c_1(t,m,n)$. Therefore, on this branch, we need definitions for the domain of the multi-linear operator, combination of the summation condition and each term in $c_1(t,m,n)$.

For the domain of the multi-linear operator, refer to \autoref{fig1} for details.
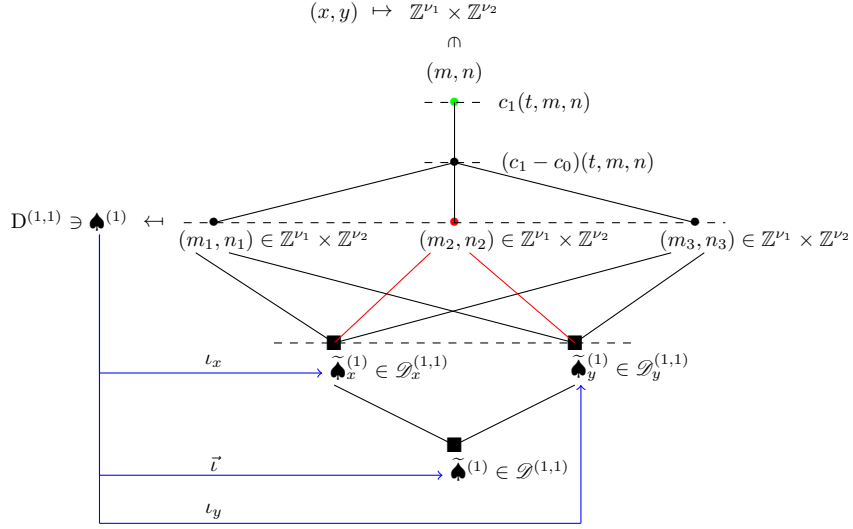
\begin{figure}[htpb!]
\centering
\begin{tikzpicture}[scale=0.8,transform shape]

\node at (2,10)   {$(x,y)$};
\node at (2.8,10) {$\mapsto$};
\node at (4,10) {$\mathbb Z^{\nu_1}\times\mathbb Z^{\nu_2}$};
\node at (4,9.5) {\rotatebox{-90}{$\in$}};
\node at (4,9) {$(m,n)$};

\node at (4,8.5) {\color{green}$\bullet$};
\draw [dashed] (3.5,8.5)--(4.5,8.5);
\node at (5.5,8.5) {$c_1(t,m,n)$};
\draw (4,8.5)--(4,7.5);
\node at (4,7.5) {$\bullet$};
\draw [dashed] (3.5,7.5)--(4.5,7.5);
\node at (6.05,7.5) {$(c_1-c_0)(t,m,n)$};

\node at (0,6.5) {$\bullet$};
\node at (4,6.5) {\color{red}$\bullet$};
\node at (8,6.5) {$\bullet$};

\node at (2,4.5) {$\blacksquare$};
\node at (6,4.5) {$\blacksquare$};

\draw (-0.3,6)--(2,4.5);
\draw [red](3.6,6)--(2,4.5);
\draw (7.6,6)--(2,4.5);

\draw (0.25,6)--(6,4.5);
\draw [red](4.25,6)--(6,4.5);
\draw (8.15,6)--(6,4.5);

\draw [dashed](1,4.5)--(7,4.5);
\node at (-2.4,6.54) {$\mathrm D^{(1,1)}\ni\spadesuit^{(1)}$};
\node at (-1,6.54) {\rotatebox{180}{$\mapsto$}};
\node at (1,6.2) {$(m_1,n_1)\in\mathbb Z^{\nu_1}\times\mathbb Z^{\nu_2}$};
\node at (5,6.2) {$(m_2,n_2)\in\mathbb Z^{\nu_1}\times\mathbb Z^{\nu_2}$};
\node at (9,6.2) {$(m_3,n_3)\in\mathbb Z^{\nu_1}\times\mathbb Z^{\nu_2}$};

\draw (0,6.5) -- (4,7.5);
\draw (4,6.5) -- (4,7.5);
\draw (8,6.5) -- (4,7.5);

\draw [dashed](-0.5,6.5)--(8.5,6.5);

\node at (2.92,4.1) {$\widetilde\spadesuit_x^{(1)}\in\mathscr D_x^{(1,1)}$};
\node at (6.92,4.1) {$\widetilde\spadesuit_y^{(1)}\in\mathscr D_y^{(1,1)}$};

\node at (4,2.8) {$\blacksquare$};

\draw (2,3.8)--(4,2.8);
\draw (6,3.8)--(4,2.8);

\node at (4.92,2.4) {$\widetilde\spadesuit^{(1)}\in\mathscr D^{(1,1)}$};

%
%
\coordinate (A) at (-1.9,6.3);
\coordinate (B) at (-1.9,4);
\coordinate (C) at (-1.9,2.3);
\coordinate (D) at (-1.9,1.5);

\coordinate (E) at (1.8,4);
\coordinate (F) at (3.8,2.3);
\coordinate (G) at (6.1,1.5);
\coordinate (H) at (6.1,3.8);

\draw [blue] (A)--(B);
\draw [->,blue] (B)--(E);
\draw [blue] (A)--(C);
\draw [->,blue] (C)--(F);
\draw [blue] (A)--(D);
\draw [blue] (D)--(G);
\draw [->,blue] (G)--(H);

\node at (0,4.2) {$\iota_x$};
\node at (0,2.5) {$\vec{\iota}$};
\node at (0,1.7) {$\iota_y$};
\end{tikzpicture}
\caption{Analysis of the nonlinear part $(c_1-c_0)(t,m,n)$ for the first iteration $c_1(t,m,n)$. The green circle {\color{green}$\bullet$}, the black circle $\bullet$, the black square $\blacksquare$ and the blue lines have the similar meanings as in \autoref{fig0}. The red points and lines are associated with the minus ``-". For example, the red circle ${\color{red}\bullet}$ means that there is a minus in front of $(m_2,n_2)$ in the summation condition $(m,n)=(m_1,n_1)-(m_2,n_2)+(m_3,n_3)$; the red lines mean that the summation condition $m=m_1-m_2+m_3$ in the $x$-direction and the summation condition $n=n_1-n_2+n_3$ in the $y$-direction.}
\label{fig1}
\end{figure}

Similar to the analysis of the linear part, there are two ways to view the domain of the multi-linear operator.

The first aspect arises from the order of $u$ appearing in the cubic nonlinearity $|u|^2u=u\bar{u}u$. The upper half of \autoref{fig1} illustrates the variation of the Fourier vector $(m,n)\in\mathbb Z^{\nu_1}\times\mathbb Z^{\nu_2}$ of the nonlinear part $(c_1-c_0)(t,m,n)$ in the first iteration $c_1(t,m,n)$. In this case, we use ${\mathrm D}^{(1,1)}$ to stand for the domain of the multi-linear operator, defined as follows:
\begin{align*}
{\mathrm D}^{(1,1)}=(\mathbb Z^{\nu_1}\times\mathbb Z^{\nu_2})\times(\mathbb Z^{\nu_1}\times\mathbb Z^{\nu_2})\times(\mathbb Z^{\nu_1}\times\mathbb Z^{\nu_2}).
\end{align*}

The second aspect concerns the space directions $x$ and $y$. That is, take the first coordinates $m_j$'s of $(m_j,n_j)$ for $j=1,2,3$, to be the triple $(m_1,m_2,m_3)\in\mathbb Z^{\nu_1}\times\mathbb Z^{\nu_1}\times\mathbb Z^{\nu_1}$; and take the second coordinates $n_j$'s of $(m_j,n_j)$ for $j=1,2,3$, to be the triple $(n_1,n_2,n_3)\in\mathbb Z^{\nu_2}\times\mathbb Z^{\nu_2}\times\mathbb Z^{\nu_2}$. Define
\begin{align*}
\mathscr D^{(1,1)}=\mathscr D_x^{(1,1)}\times \mathscr D_y^{(1,1)},
\end{align*}
where
\begin{align*}
\mathscr D_x^{(1,1)}=\mathbb Z^{\nu_1}\times\mathbb Z^{\nu_1}\times\mathbb Z^{\nu_1};\quad
\mathscr D_y^{(1,1)}=\mathbb Z^{\nu_2}\times\mathbb Z^{\nu_2}\times\mathbb Z^{\nu_2}.
\end{align*}

Clearly, there is a natural isomorphism $\vec{\iota}=(\iota_x,\iota_y)$ between ${\mathrm D}^{(1,1)}$ and ${\mathscr D}^{(1,1)}$, that is,
\begin{align*}
\vec{\iota}: ~{\mathrm D}^{(1,1)}&\rightarrow{\mathscr D}^{(1,1)}\nonumber\\
\underbrace{\big((m_1,n_1),(m_2,n_2),(m_3,n_3)\big)}_{\in(\mathbb Z^{\nu_1}\times\mathbb Z^{\nu_2})^3}=\spadesuit^{(1)}
&\mapsto\widetilde\spadesuit^{(1)}=\vec{\iota}(\spadesuit^{(1)})\\
&\hspace{13.4mm}=\left(\iota_x(\spadesuit^{(1)}),\iota_y(\spadesuit^{(1)})\right)\\
&\hspace{13.4mm}=\left(\widetilde\spadesuit_{x}^{(1)},\widetilde\spadesuit_{y}^{(1)}\right)\\
&\hspace{13.4mm}=\Big((m_1,m_2,m_3),(n_1,n_2,n_3)\Big).
\end{align*}

As to the combination of the summation condition $(m,n)=\sum_{j=1}^{3}(-1)^{j-1}(m_j,n_j)$, define
\begin{align*}
\lambda(\spadesuit^{(1)})&=(\vec{\cc\ca\cs}\circ\vec{\iota})(\spadesuit^{(1)})\\
&=\left((\cc\ca\cs_x\circ\iota_x)(\spadesuit^{(1)}),(\cc\ca\cs_y\circ\iota_y)(\spadesuit^{(1)})\right)\\
&=\left(\cc\ca\cs_x(\widetilde\spadesuit_x^{(1)}),\cc\ca\cs_y(\widetilde\spadesuit_y^{(1)})\right)\\
&=\left(\cc\ca\cs_x(m_1,m_2,m_3),\cc\ca\cs_y(n_1,n_2,n_3)\right)\\
&=(m_1-m_2+m_3,n_1-n_2+n_3)\nonumber\\
&=(m,n).\nonumber
\end{align*}

Furthermore, on the branch $1\in\Gamma^{(1)}$, define
\begin{align*}
\mathfrak C^{(1,1)}(\spadesuit^{(1)})&=\prod_{j=1}^3\{c(m_j,n_j)\}^{\ast^{[j-1]}};\\
\mathfrak I^{(1,1)}(t,\spadesuit^{(1)})&=\int_0^t\Phi^{t-s}(\lambda(\spadesuit^{(1)}))\prod_{j=1}^{3}\left\{\Phi^s(m_j,n_j)\right\}^{\ast^{[j-1]}}{\rm d}s;\\
\mathfrak F^{(1,1)}(\spadesuit^{(1)})&={\rm i}\varepsilon.
\end{align*}

This shows that
\begin{align}\label{c0n}
&(c_1-c_0)(t,m,n)\nonumber\\
=~&\sum_{\gamma^{(1)}=1\in\Gamma^{(1)}}
\sum_{\substack{\lambda(\spadesuit^{(1)})=(m,n)\\\spadesuit^{(1)}\in{\mathrm D}^{(1,\gamma^{(1)})}}}
\mathfrak C^{(1,\gamma^{(1)})}(\spadesuit^{(1)})\mathfrak I^{(1,\gamma^{(1)})}(t,\spadesuit^{(1)})\mathfrak F^{(1,\gamma^{(1)})}(\spadesuit^{(1)}).
\end{align}

Combining \eqref{c0l} with \eqref{c0n} yields
\begin{align*}
c_1(t,m,n)
=~&c_0(t,m,n)+(c_1-c_0)(t,m,n)\nonumber\\
=~&\sum_{\gamma^{(1)}\in\Gamma^{(1)}}\sum_{\substack{\lambda(\spadesuit^{(1)})=(m,n)\\\spadesuit^{(1)}\in{\mathrm D}^{(1,\gamma^{(1)})}}}
\mathfrak C^{(1,\gamma^{(1)})}(\spadesuit^{(1)})\mathfrak I^{(1,\gamma^{(1)})}(t,\spadesuit^{(1)})\mathfrak F^{(1,\gamma^{(1)})}(\spadesuit^{(1)}).\label{c1}
\end{align*}
This concludes the analysis of the first Picard iteration.

\begin{remark}\label{re1}
Before starting the next iteration, we describe the number of components of the elements in  $\mathscr D_x^{(k,0)}(k\geq1), \mathscr D_x^{(1,\gamma^{(1)})}$ according to $\mathbb Z^{\nu_1}$, and those in $\mathscr D_y^{(k,0)}(k\geq1), \mathscr D_y^{(1,\gamma^{(1)})}$ according to $\mathbb Z^{\nu_2}$, where $\gamma^{(1)}\in\Gamma^{(1)}$. It follows from \cite{DLX1} that
\begin{align*}
\mathscr D_x^{(k,0)}&=(\mathbb Z^{\nu_1})^{2\sigma(0)},\quad
\mathscr D_y^{(k,0)}=(\mathbb Z^{\nu_2})^{2\sigma(0)},\quad\text{for}~k\geq1;
\end{align*}
and
\begin{align*}
\mathscr D_x^{(1,\gamma^{(1)})}&=(\mathbb Z^{\nu_1})^{2\sigma(\gamma^{(1)})},\quad
\mathscr D_y^{(1,\gamma^{(1)})}=(\mathbb Z^{\nu_2})^{2\sigma(\gamma^{(1)})},
\end{align*}
where $\sigma$ denotes the first-class counting function, defined as follows:
\begin{align}
\sigma(\gamma^{(k)})=
\begin{cases}
\frac{1}{2},&\gamma^{(k)}=0\in\Gamma^{(0)}, k\geq1;\\
\frac{3}{2},&\gamma^{(1)}=1\in\Gamma^{(1)}.
\end{cases}
\end{align}
\end{remark}

%
%
{\bf Step 4. Analysis of the Second Iteration}~

For the definitions of $\mathrm D^{(k,\gamma^{(k)})}, \lambda, \mathfrak C,\mathfrak I$ and $\mathfrak F$ on the branch $0\in\Gamma^{(k)}$, we refer to {\bf Step 2}.

We now begin analyzing the nonlinear part $(c_2-c_0)(t,m,n)$ of the second Picard iteration. For $k=2$, it follows from \eqref{pi} and \eqref{c1} that
\begin{align*}
&(c_2-c_0)(t,m,n)\\
=~&{\rm i}\varepsilon\int_0^t\Phi^{t-s}(m,n)
\sum_{\substack{(m,n)=\sum_{j=1}^3(-1)^{j-1}(m_j,n_j)\\(m_j,n_j)\in\mathbb Z^{\nu_1}\times\mathbb Z^{\nu_2}\\j=1,2,3}}\prod_{j=1}^{3}\left\{c_{1}(s,m,n)\right\}^{\ast^{[j-1]}}{\rm d}s\\
=~&{\rm i}\varepsilon\int_0^t{\rm d}s\Phi^{t-s}(m,n)\sum_{\substack{(m,n)=\sum_{j=1}^3(-1)^{j-1}(m_j,n_j)\\(m_j,n_j)\in\mathbb Z^{\nu_1}\times\mathbb Z^{\nu_2}\\j=1,2,3}}\\
&
\prod_{j=1}^3\left\{\sum_{\gamma_j^{(1)}\in\Gamma^{(1)}}
\sum_{\substack{\lambda(\spadesuit_j^{(1)})=(m_j,n_j)\\\spadesuit_j^{(1)}\in{\mathrm D}^{(1,\gamma_j^{(1)})}}}
\mathfrak C^{(1,\gamma_j^{(1)})}(\spadesuit^{(1)})\mathfrak I^{(1,\gamma_j^{(1)})}(s,\spadesuit_j^{(1)})\mathfrak F^{(1,\gamma_j^{(1)})}(\spadesuit_j^{(1)})\right\}^{\ast^{[j-1]}}\\
=~&\sum_{\substack{\gamma_j^{(1)}\in\Gamma^{(1)}\\j=1,2,3}}
\sum_{\substack{(m,n)=\sum_{j=1}^3(-1)^{j-1}(m_j,n_j)\\(m_j,n_j)\in\mathbb Z^{\nu_1}\times\mathbb Z^{\nu_2}\\j=1,2,3}}
\sum_{\substack{\lambda(\spadesuit_j^{(1)})=(m_j,n_j)\\\spadesuit_j^{(1)}\in{\mathrm D}^{(1,\gamma_j^{(1)})}\\j=1,2,3}}\prod_{j=1}^3\left\{\mathfrak C^{(1,\gamma_j^{(1)})}(\spadesuit_j^{(1)})\right\}^{\ast^{[j-1]}}\\
&\times\int_0^t\Phi^{t-s}(m,n)\prod_{j=1}^{3}\left\{\mathfrak I^{(1,\gamma_j^{(1)})}(s,\spadesuit_j^{(1)})\right\}^{\ast^{[j-1]}}\times{\rm i}\varepsilon\left\{\mathfrak F^{(1,\gamma_j^{(1)})}(\spadesuit_j^{(1)})\right\}^{\ast^{[j-1]}}.
\end{align*}

Similar to the analysis of the first iteration, we aim to use $\Gamma^{(1)}\times\Gamma^{(1)}\times\Gamma^{(1)}$ to label the nonlinear part $(c_2-c_0)(t,m,n)$ of the second iteration $c_2(t,m,n)$. On this branch, with the aid of \autoref{fig3}, we introduce definitions for the domain of the multi-linear operator, the combination of the summation condition and each term in $c_2(t,m,n)$.

\begin{figure}[htpb!]
\centering
\begin{tikzpicture}[scale=0.7,transform shape]
\node at (5.15,9.2) {$(x,y)$};
\node at (5.88,9.2) {$\mapsto$};
\node at (7,9.2)   {{$\mathbb Z^{\nu_1}\times\mathbb Z^{\nu_2}$}};
\node at (7,8.8)   {\rotatebox{90}{$\in$}};
\node at (7,8.4)   {$(m,n)$};
\node at (7,8)   {\color{green}$\bullet$};
\draw [dashed] (6.5,8)--(7.5,8);
\node at (8.65,8) {$c_2(t,m,n)$};

\node at (6.4,7.3) {$(m,n)$};
\node at (7,7)   {\color{red}$\bullet$};
\draw [dashed] (6.5,7)--(7.5,7);
\node at (9.2,7) {$(c_2-c_0)(t,m,n)$};

\node at (-0.25,6) {$\mathrm D^{(2,\gamma^{(2)})}\ni\spadesuit^{(2)}$};
\node at (1.2,6.02) {\rotatebox{-180}{$\mapsto$}};
\node at (2,6)   {$\bullet$};
\node at (3.8,5.7) {$(m_1,n_1)\in\mathbb Z^{\nu_1}\times\mathbb Z^{\nu_2}$};
\node at (7,6)   {\color{red}$\bullet$};
\node at (8.8,5.7) {$(m_2,n_2)\in\mathbb Z^{\nu_1}\times\mathbb Z^{\nu_2}$};
\node at (12,6)   {$\bullet$};
\node at (13.8,5.7) {$(m_2,n_2)\in\mathbb Z^{\nu_1}\times\mathbb Z^{\nu_2}$};
\draw [dashed] (1.5,6)--(12.5,6);

\node at (2,5)   {$\bullet$};
\node at (3.35,4.7) {$\spadesuit^{(1)}_1\in\mathrm D^{(1,\gamma_1^{(1)})}$};
\node at (7,5)   {\color{red}$\bullet$};
\node at (8.35,4.7) {$\spadesuit^{(1)}_2\in\mathrm D^{(1,\gamma_2^{(1)})}$};
\node at (12,5)   {$\bullet$};
\node at (13.35,4.7) {$\spadesuit^{(1)}_3\in\mathrm D^{(1,\gamma_3^{(1)})}$};
\draw [dashed] (1.5,5)--(12.5,5);

\node at (2,4)   {$\bullet$};
\node at (3.65,3.7)   {$(\mathbb Z^{\nu_1}\times\mathbb Z^{\nu_2})^{2\sigma(\gamma_1^{(1)})}$};
\node at (3.1,3.2)   {\rotatebox{90}{$\in$}};
\node at (4,2.7) {$\left((m_j^1,n_j^1)\right)_{1\leq j\leq 2\sigma(\gamma_1^{(1)})}$};

\node at (7,4)   {\color{red}$\bullet$};
\node at (8.65,3.7)   {$(\mathbb Z^{\nu_1}\times\mathbb Z^{\nu_2})^{2\sigma(\gamma_2^{(1)})}$};
\node at (8.1,3.2)   {\rotatebox{90}{$\in$}};
\node at (9,2.7) {$\left((m_j^2,n_j^2)\right)_{1\leq j\leq 2\sigma(\gamma_2^{(1)})}$};

\node at (12,4)   {$\bullet$};
\node at (13.65,3.7)   {$(\mathbb Z^{\nu_1}\times\mathbb Z^{\nu_2})^{2\sigma(\gamma_3^{(1)})}$};
\node at (13.1,3.2)   {\rotatebox{90}{$\in$}};
\node at (14,2.7) {$\left((m_j^3,n_j^3)\right)_{1\leq j\leq 2\sigma(\gamma_3^{(1)})}$};

\draw [dashed] (1.5,4)--(12.5,4);

\node at (3.25,0)   {$\blacksquare$};
\draw (2.8,2.5)--(3.25,0);
\draw (3.4,2.5)--(8.25,0);
\draw [red](7.8,2.5)--(4.5,0);
\draw [red](8.4,2.5)--(9.5,0);
\draw (12.8,2.5)--(5.75,0);
\draw (13.4,2.5)--(10.75,0);

\node at (4.5,0)   {\color{red}$\blacksquare$};
\node at (5.75,0)  {$\blacksquare$};

\node at (8.25,0)   {$\blacksquare$};
\node at(9.5,0)   {\color{red}$\blacksquare$};
\node at (10.75,0)  {$\blacksquare$};
\draw [dashed] (2.6,0)--(11.5,0);

\draw (7,8)--(7,7);
\draw (7,7)--(2,6);
\draw (7,7)--(7,6);
\draw (7,7)--(12,6);
\draw (2,5)--(2,6);
\draw (2,5)--(2,4);
\draw (7,5)--(7,6);
\draw (7,5)--(7,4);
\draw (12,5)--(12,6);
\draw (12,5)--(12,4);

\node at (3.45,-0.5) {$\widetilde\spadesuit_{1,x}^{(1)}$};
\node at (4.7,-0.5) {$\widetilde\spadesuit_{2,x}^{(1)}$};
\node at (5.95,-0.5) {$\widetilde\spadesuit_{3,x}^{(1)}$};

\draw (3.25,-0.7)--(4.5,-2);
\draw [red](4.5,-0.7)--(4.5,-2);
\draw (5.75,-0.7)--(4.5,-2);

\draw (8.25,-0.7)--(9.5,-2);
\draw [red](9.5,-0.7)--(9.5,-2);
\draw (10.75,-0.7)--(9.5,-2);

\node at (8.45,-0.5) {$\widetilde\spadesuit_{1,y}^{(1)}$};
\node at (9.7,-0.5) {$\widetilde\spadesuit_{2,y}^{(1)}$};
\node at (10.95,-0.5) {$\widetilde\spadesuit_{3,y}^{(1)}$};

\node at (4.5,-2) {$\blacksquare$};
\node at (5.6,-2.5) {$\widetilde\spadesuit_x^{(2)}\in\mathscr D_x^{(2,\gamma^{(2)})}$};
\node at (9.5,-2) {$\blacksquare$};
\node at (10.6,-2.5) {$\widetilde\spadesuit_y^{(2)}\in\mathscr D_y^{(2,\gamma^{(2)})}$};
\draw [dashed](4,-2)--(10,-2);

\node at (7,-4) {$\blacksquare$};
\node at (8.1,-4.4) {$\spadesuit^{(2)}\in\mathscr D^{(2,\gamma^{(2)})}$};

\draw (7,-4)--(4.5,-2.8);
\draw (7,-4)--(9.5,-2.8);

 \coordinate (A) at (0.42,5.7);
 \coordinate (B) at (0.42,-2.6);
 \coordinate (C) at (0.42,-4.5);
 \coordinate (D) at (4.3,-2.6);
 \coordinate (E) at (6.8,-4.5);
 \coordinate (F) at (0.42,-5.5);
 \coordinate (G) at (9.6,-5.5);
 \coordinate (H) at (9.6,-2.8);
\draw [blue](A)--(B);
\draw [->,blue](B)--(D);
\draw [blue](A)--(C);
\draw [->,blue](C)--(E);
\draw [blue](A)--(F);
\draw [blue](F)--(G);
\draw [->,blue](G)--(H);

\node at (2.4,-2.4) {$\iota_x$};
\node at (2.4,-4.3) {$\vec{\iota}$};
\node at (2.4,-5.3) {$\iota_y$};
\end{tikzpicture}
\caption{Analysis of the nonlinear part $(c_2-c_0)(t,m,n)$ of the second iteration $c_2(t,m,n)$. The same elements in \autoref{fig3} are analogous to those in \autoref{fig1}. In addition, the red square ${\color{red}\blacksquare}$ is also assigned with a minus sign.}
\label{fig3}
\end{figure}
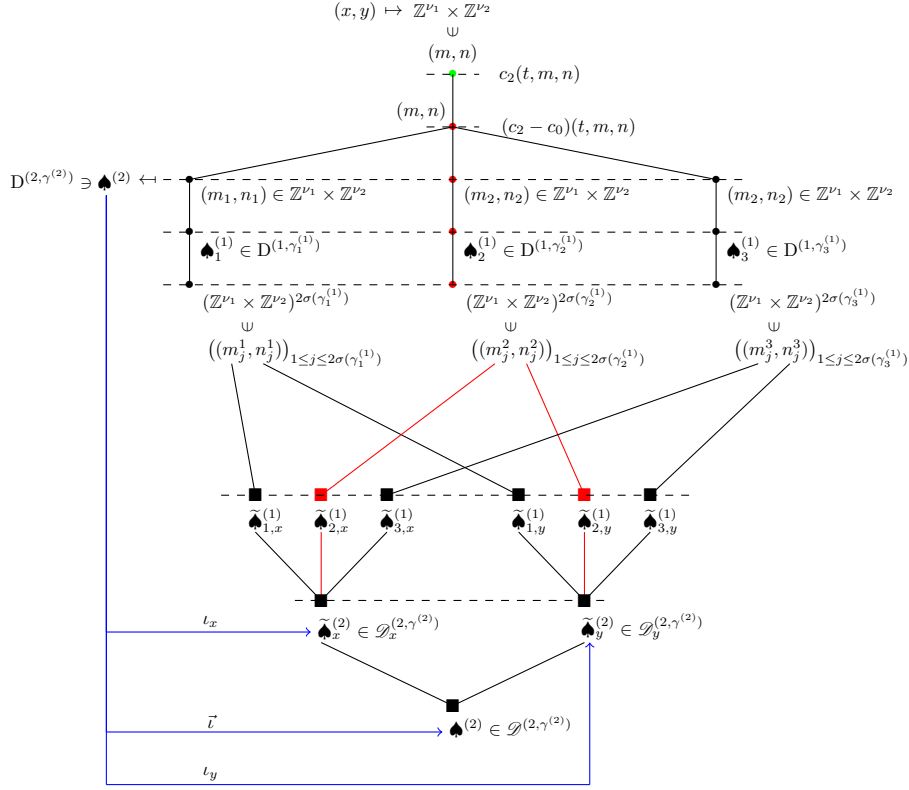

For $\gamma^{(2)}=(\gamma_1^{(1)},\gamma_2^{(2)},\gamma_3^{(1)})\in\Gamma^{(1)}
\times\Gamma^{(1)}\times\Gamma^{(1)}\subset\Gamma^{(2)}$,
define
\begin{align*}
\prod_{j=1}^3{\mathrm D}^{(1,\gamma_j^{(1)})}={\mathrm D}^{(2,\gamma^{(2)})}
\stackrel{\vec{\iota}}{\boldsymbol\simeq}
\mathscr D^{(2,\gamma^{(2)})}=\prod_{j=1}^3\mathscr D_x^{(1,\gamma_j^{(1)})}\times\prod_{j=1}^3\mathscr D_y^{(1,\gamma_j^{(1)})}.
\end{align*}
The key issue is to understand the map $\vec{\iota}$. In this context,
we use $\spadesuit^{(2)}$ to denote an element in $\mathrm D^{(2,\gamma^{(2)})}$, which consists of $(\spadesuit_1^{(1)},\spadesuit_2^{(1)},\spadesuit_3^{(1)})\in\prod_{j=1}^3{\mathrm D}^{(1,\gamma_j^{(1)})}$, where
\begin{align*}
\spadesuit_1^{(1)}&=\big((m_j^1,n_j^1)\big)_{1\leq j\leq 2\sigma(\gamma_1^{(1)})}\in{\mathrm D}^{(1,\gamma_1^{(1)})};\\
\spadesuit_2^{(1)}&=\big((m_j^2,n_j^2)\big)_{1\leq j\leq 2\sigma(\gamma_2^{(1)})}\in{\mathrm D}^{(1,\gamma_2^{(1)})};\\
\spadesuit_3^{(1)}&=\big((m_j^3,n_j^3)\big)_{1\leq j\leq 2\sigma(\gamma_3^{(1)})}\in{\mathrm D}^{(1,\gamma_3^{(1)})}.
\end{align*}
Grouping them according to spatial directions, we have $\widetilde\spadesuit^{(2)}=(\widetilde\spadesuit_x^{(2)},\widetilde\spadesuit_y^{(2)})$, where $\widetilde\spadesuit_x^{(2)}=(\widetilde\spadesuit_{1,x}^{(1)},\widetilde\spadesuit_{2,x}^{(1)},\widetilde\spadesuit_{3,x}^{(1)})$
,
\begin{align*}
\widetilde\spadesuit_{1,x}^{(1)}&=
(m_j^1)_{1\leq j\leq 2\sigma(\gamma_1^{(1)})}\in{\mathscr D}_x^{(1,\gamma_1^{(1)})};\\
\widetilde\spadesuit_{2,x}^{(1)}&=(m_j^2)_{1\leq j\leq 2\sigma(\gamma_2^{(1)})}\in{\mathscr D}_x^{(1,\gamma_2^{(1)})};\\
\widetilde\spadesuit_{3,x}^{(1)}&=(m_j^3)_{1\leq j\leq 2\sigma(\gamma_3^{(1)})}\in{\mathscr D}_x^{(1,\gamma_3^{(1)})},
\end{align*}
and $\widetilde\spadesuit_y^{(2)}=
(\widetilde\spadesuit_{1,y}^{(1)},\widetilde\spadesuit_{2,y}^{(1)},\widetilde\spadesuit_{3,y}^{(1)})$,
\begin{align*}
\widetilde\spadesuit_{1,y}^{(1)}&=
(n_j^1)_{1\leq j\leq 2\sigma(\gamma_1^{(1)})}\in{\mathscr D}_y^{(1,\gamma_1^{(1)})};\\
\widetilde\spadesuit_{2,y}^{(1)}&=(n_j^2)_{1\leq j\leq 2\sigma(\gamma_2^{(1)})}\in{\mathscr D}_y^{(1,\gamma_2^{(1)})};\\
\widetilde\spadesuit_{3,y}^{(1)}&=(n_j^3)_{1\leq j\leq 2\sigma(\gamma_3^{(1)})}\in{\mathscr D}_y^{(1,\gamma_3^{(1)})}.
\end{align*}
That is,
\begin{align*}
\vec{\iota}(\spadesuit^{(2)})&=\left(\iota_x(\spadesuit^{(2)}),\iota_y(\spadesuit^{(2)})\right)\\
&=\left(\widetilde\spadesuit_x^{(2)},\widetilde\spadesuit_y^{(2)}\right)\\
&=\left((\widetilde\spadesuit_{1,x}^{(1)},\widetilde\spadesuit_{2,x}^{(1)},\widetilde\spadesuit_{3,x}^{(1)}),
(\widetilde\spadesuit_{1,y}^{(1)},\widetilde\spadesuit_{2,y}^{(1)},\widetilde\spadesuit_{3,y}^{(1)})\right)\\
&=\Big((m_j^1)_{1\leq j\leq2\sigma(\gamma_1^{(1)})},(m_j^2)_{1\leq j\leq2\sigma(\gamma_2^{(1)})},(m_j^3)_{1\leq j\leq2\sigma(\gamma_3^{(1)})},\\
&\hspace{6.75mm}(n_j^1)_{1\leq j\leq2\sigma(\gamma_1^{(1)})},(n_j^2)_{1\leq j\leq2\sigma(\gamma_2^{(1)})},(n_j^3)_{1\leq j\leq2\sigma(\gamma_3^{(1)})}\Big).
\end{align*}

Define
\begin{align*}
\lambda(\spadesuit^{(2)})
&=(\vec{\cc\ca\cs}\circ\vec{\iota})(\spadesuit^{(2)})\\
&=\left((\cc\ca\cs_x\circ\iota_x)(\spadesuit^{(2)}),
(\cc\ca\cs_y\circ\iota_y)(\spadesuit^{(2)})
\right)\\
&=\left(\cc\ca\cs_x(\widetilde\spadesuit_x^{(2)}),\cc\ca\cs_y(\widetilde\spadesuit_y^{(2)})\right)\\
&=\left(\cc\ca\cs_x(\widetilde\spadesuit_{1,x}^{(1)},
\widetilde\spadesuit_{2,x}^{(1)},
\widetilde\spadesuit_{3,x}^{(1)}),
\cc\ca\cs_y(\widetilde\spadesuit_{1,y}^{(1)},
\widetilde\spadesuit_{2,y}^{(1)},
\widetilde\spadesuit_{3,y}^{(1)})
\right)\\
&=\left(\sum_{j=1}^3(-1)^{j-1}\widetilde\spadesuit_{j,x}^{(1)},
\sum_{j=1}^3(-1)^{j-1}\widetilde\spadesuit_{j,y}^{(1)}\right)\\
&=\sum_{j=1}^{3}(-1)^{j-1}\left(\widetilde\spadesuit_{j,x}^{(1)},\widetilde\spadesuit_{j,y}^{(1)}\right)\\
&=\sum_{j=1}^{3}(-1)^{j-1}\left((\cc\ca\cs_x\circ\iota_x)(\spadesuit_j^{(1)}),
(\cc\ca\cs_y\circ\iota_y)(\spadesuit_j^{(1)})\right)\\
&=\sum_{j=1}^{3}(-1)^{j-1}\left(\vec{\cc\ca\cs}(\spadesuit_j^{(1)})\right)\\
&=\sum_{j=1}^3(-1)^{j-1}\lambda(\spadesuit_j^{(1)}).
\end{align*}

Furthermore, define
\begin{align*}
\mathfrak C^{(2,\gamma^{(2)})}(\spadesuit^{(2)})&=\prod_{j=1}^3\left\{\mathfrak C^{(1,\gamma_j^{(1)})}(\spadesuit_j^{(1)})\right\}^{\ast^{[j-1]}};\\
\mathfrak I^{(2,\gamma^{(2)})}(t,\spadesuit^{(2)})&=\int_0^t\Phi^{t-s}
(\lambda(\spadesuit^{(2)}))\prod_{j=1}^{3}\left\{\mathfrak I^{(1,\gamma_j^{(1)})}(s,\lambda(\spadesuit_j^{(1)}))\right\}^{\ast^{[j-1]}}{\rm d}s;\\
\mathfrak F^{(2,\gamma^{(2)})}(\spadesuit^{(2)})&={\rm i}\varepsilon\prod_{j=1}^3\left\{\mathfrak F^{(1,\gamma_j^{(1)})}(\spadesuit_j^{(1)})\right\}^{\ast^{[j-1]}}.
\end{align*}

Hence, we obtain
\begin{align}\label{c2n}
&(c_2-c_0)(t,m,n)\nonumber\\
=~&\sum_{\gamma^{(2)}\in\Gamma^{(1)}\times\Gamma^{(1)}\times\Gamma^{(1)}}
\sum_{\substack{\lambda(\spadesuit^{(2)})=(m,n)\\\spadesuit^{(2)}\in{\mathrm D}^{(2,\gamma^{(2)})}}}
\mathfrak C^{(2,\gamma^{(2)})}(\spadesuit^{(2)})
\mathfrak I^{(2,\gamma^{(2)})}(t,\spadesuit^{(2)})
\mathfrak F^{(2,\gamma^{(2)})}(\spadesuit^{(2)}).
\end{align}

Putting \eqref{c2n} and \eqref{c0l} yields
\begin{align*}
c_2(t,m,n)&=c_0(t,m,n)+(c_2-c_0)(t,m,n)\nonumber\\
&=\sum_{\gamma^{(2)}\in\Gamma^{(2)}}\sum_{\substack{\lambda(\spadesuit^{(2)})=(m,n)\\\spadesuit^{(1)}\in{\mathrm D}^{(2,\gamma^{(2)})}}}
\mathfrak C^{(2,\gamma^{(2)})}(\spadesuit^{(2)})
\mathfrak I^{(2,\gamma^{(2)})}(t,\spadesuit^{(2)})
\mathfrak F^{(2,\gamma^{(2)})}(\spadesuit^{(2)}).\label{c1}
\end{align*}

This completes the analysis of the second Picard iteration.

\begin{remark}\label{re2}
Similar to the discussion in \autoref{re1}, let $\gamma^{(2)}=(\gamma_1^{(1)},\gamma_2^{(1)},\gamma_3^{(1)})\in\Gamma^{(1)}\times\Gamma^{(1)}\times\Gamma^{(1)}$.
It follows from \cite{DLX1} that
\begin{align*}
\mathscr D_x^{(1,\gamma^{(2)})}=(\mathbb Z^{\nu_1})^{2\sigma(\gamma^{(2)})},\quad
\mathscr D_y^{(1,\gamma^{(2)})}=(\mathbb Z^{\nu_2})^{2\sigma(\gamma^{(2)})},
\end{align*}
where $\sigma$ stands for the first-class counting function defined by letting
\begin{align}
\sigma(\gamma^{(2)})=\sum_{j=1}^3\sigma(\gamma_j^{(1)}).
\end{align}
\end{remark}
{\bf Step 5. Analysis of the General Case.}~

In this step, we shall analyze the $k$-th iteration and prove that, for all $k\geq2$,
\begin{align}\label{ke}
c_k(t,m,n)=\sum_{\gamma^{(k)}\in\Gamma^{(k)}}
\sum_{\substack{\spadesuit^{(k)}\in{\mathrm D}^{(k,\gamma^{(k)})}\\\lambda(\spadesuit^{(k)})=(m,n)}}
\mathfrak C^{(k,\gamma^{(k)})}(\spadesuit^{(k)})
\mathfrak I^{(k,\gamma^{(k)})}(t,\spadesuit^{(k)})
\mathfrak F^{(k,\gamma^{(k)})}(\spadesuit^{(k)}).
\end{align}

Below we give the introduction to the domain ${\mathrm D}^{(k,\gamma^{(k)})}$, the summation condition $\lambda$, the terms $\mathfrak C,\mathfrak I$ and $\mathfrak F$.

For the definitions of $\mathrm D^{(k,\gamma^{(k)})}, \lambda, \mathfrak C,\mathfrak I$ and $\mathfrak F$ on the branch $0\in\Gamma^{(k)}$, we refer to {\bf Step 2}.


With the help of \autoref{fig4}, we introduce them as follows.
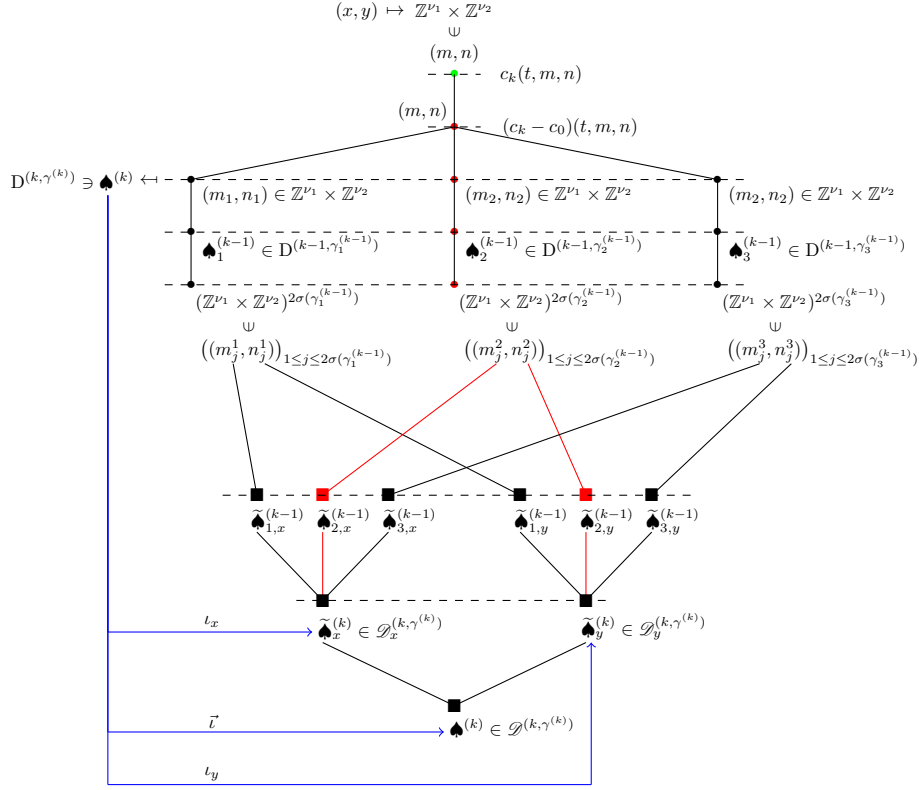
\begin{figure}[htpb!]
\centering
\begin{tikzpicture}[scale=0.7,transform shape]
\node at (5.15,9.2) {$(x,y)$};
\node at (5.88,9.2) {$\mapsto$};
\node at (7,9.2)   {{$\mathbb Z^{\nu_1}\times\mathbb Z^{\nu_2}$}};
\node at (7,8.8)   {\rotatebox{90}{$\in$}};
\node at (7,8.4)   {$(m,n)$};
\node at (7,8)   {\color{green}$\bullet$};
\draw [dashed] (6.5,8)--(7.5,8);
\node at (8.65,8) {$c_k(t,m,n)$};

\node at (6.4,7.3) {$(m,n)$};
\node at (7,7)   {\color{red}$\bullet$};
\draw [dashed] (6.5,7)--(7.5,7);
\node at (9.2,7) {$(c_k-c_0)(t,m,n)$};

\node at (-0.25,6) {$\mathrm D^{(k,\gamma^{(k)})}\ni\spadesuit^{(k)}$};
\node at (1.2,6.02) {\rotatebox{-180}{$\mapsto$}};
\node at (2,6)   {$\bullet$};
\node at (3.8,5.7) {$(m_1,n_1)\in\mathbb Z^{\nu_1}\times\mathbb Z^{\nu_2}$};
\node at (7,6)   {\color{red}$\bullet$};
\node at (8.8,5.7) {$(m_2,n_2)\in\mathbb Z^{\nu_1}\times\mathbb Z^{\nu_2}$};
\node at (12,6)   {$\bullet$};
\node at (13.8,5.7) {$(m_2,n_2)\in\mathbb Z^{\nu_1}\times\mathbb Z^{\nu_2}$};
\draw [dashed] (1.5,6)--(12.5,6);

\node at (2,5)   {$\bullet$};
\node at (3.9,4.7) {$\spadesuit^{(k-1)}_1\in\mathrm D^{(k-1,\gamma_1^{(k-1)})}$};
\node at (7,5)   {\color{red}$\bullet$};
\node at (8.9,4.7) {$\spadesuit^{(k-1)}_2\in\mathrm D^{(k-1,\gamma_2^{(k-1)})}$};
\node at (12,5)   {$\bullet$};
\node at (13.9,4.7) {$\spadesuit^{(k-1)}_3\in\mathrm D^{(k-1,\gamma_3^{(k-1)})}$};
\draw [dashed] (1.5,5)--(12.5,5);

\node at (2,4)   {$\bullet$};
\node at (3.65,3.7)   {$(\mathbb Z^{\nu_1}\times\mathbb Z^{\nu_2})^{2\sigma(\gamma_1^{(k-1)})}$};
\node at (3.1,3.2)   {\rotatebox{90}{$\in$}};
\node at (4,2.7) {$\left((m_j^1,n_j^1)\right)_{1\leq j\leq 2\sigma(\gamma_1^{(k-1)})}$};

\node at (7,4)   {\color{red}$\bullet$};
\node at (8.65,3.7)   {$(\mathbb Z^{\nu_1}\times\mathbb Z^{\nu_2})^{2\sigma(\gamma_2^{(k-1)})}$};
\node at (8.1,3.2)   {\rotatebox{90}{$\in$}};
\node at (9,2.7) {$\left((m_j^2,n_j^2)\right)_{1\leq j\leq 2\sigma(\gamma_2^{(k-1)})}$};

\node at (12,4)   {$\bullet$};
\node at (13.65,3.7)   {$(\mathbb Z^{\nu_1}\times\mathbb Z^{\nu_2})^{2\sigma(\gamma_3^{(k-1)})}$};
\node at (13.1,3.2)   {\rotatebox{90}{$\in$}};
\node at (14,2.7) {$\left((m_j^3,n_j^3)\right)_{1\leq j\leq 2\sigma(\gamma_3^{(k-1)})}$};

\draw [dashed] (1.5,4)--(12.5,4);

\node at (3.25,0)   {$\blacksquare$};
\draw (2.8,2.5)--(3.25,0);
\draw (3.4,2.5)--(8.25,0);
\draw [red](7.8,2.5)--(4.5,0);
\draw [red](8.4,2.5)--(9.5,0);
\draw (12.8,2.5)--(5.75,0);
\draw (13.4,2.5)--(10.75,0);

\node at (4.5,0)   {\color{red}$\blacksquare$};
\node at (5.75,0)  {$\blacksquare$};

\node at (8.25,0)   {$\blacksquare$};
\node at(9.5,0)   {\color{red}$\blacksquare$};
\node at (10.75,0)  {$\blacksquare$};
\draw [dashed] (2.6,0)--(11.5,0);

\draw (7,8)--(7,7);
\draw (7,7)--(2,6);
\draw (7,7)--(7,6);
\draw (7,7)--(12,6);
\draw (2,5)--(2,6);
\draw (2,5)--(2,4);
\draw (7,5)--(7,6);
\draw (7,5)--(7,4);
\draw (12,5)--(12,6);
\draw (12,5)--(12,4);

\node at (3.65,-0.5) {$\widetilde\spadesuit_{1,x}^{(k-1)}$};
\node at (4.9,-0.5) {$\widetilde\spadesuit_{2,x}^{(k-1)}$};
\node at (6.15,-0.5) {$\widetilde\spadesuit_{3,x}^{(k-1)}$};

\draw (3.25,-0.7)--(4.5,-2);
\draw [red](4.5,-0.7)--(4.5,-2);
\draw (5.75,-0.7)--(4.5,-2);

\draw (8.25,-0.7)--(9.5,-2);
\draw [red](9.5,-0.7)--(9.5,-2);
\draw (10.75,-0.7)--(9.5,-2);

\node at (8.65,-0.5) {$\widetilde\spadesuit_{1,y}^{(k-1)}$};
\node at (9.9,-0.5) {$\widetilde\spadesuit_{2,y}^{(k-1)}$};
\node at (11.15,-0.5) {$\widetilde\spadesuit_{3,y}^{(k-1)}$};

\node at (4.5,-2) {$\blacksquare$};
\node at (5.6,-2.5) {$\widetilde\spadesuit_x^{(k)}\in\mathscr D_x^{(k,\gamma^{(k)})}$};
\node at (9.5,-2) {$\blacksquare$};
\node at (10.6,-2.5) {$\widetilde\spadesuit_y^{(k)}\in\mathscr D_y^{(k,\gamma^{(k)})}$};
\draw [dashed](4,-2)--(10,-2);

\node at (7,-4) {$\blacksquare$};
\node at (8.1,-4.4) {$\spadesuit^{(k)}\in\mathscr D^{(k,\gamma^{(k)})}$};

\draw (7,-4)--(4.5,-2.8);
\draw (7,-4)--(9.5,-2.8);

 \coordinate (A) at (0.42,5.7);
 \coordinate (B) at (0.42,-2.6);
 \coordinate (C) at (0.42,-4.5);
 \coordinate (D) at (4.3,-2.6);
 \coordinate (E) at (6.8,-4.5);
 \coordinate (F) at (0.42,-5.5);
 \coordinate (G) at (9.6,-5.5);
 \coordinate (H) at (9.6,-2.8);
\draw [blue](A)--(B);
\draw [->,blue](B)--(D);
\draw [blue](A)--(C);
\draw [->,blue](C)--(E);
\draw [blue](A)--(F);
\draw [blue](F)--(G);
\draw [->,blue](G)--(H);

\node at (2.4,-2.4) {$\iota_x$};
\node at (2.4,-4.3) {$\vec{\iota}$};
\node at (2.4,-5.3) {$\iota_y$};
\end{tikzpicture}
\caption{Analysis of the nonlinear part $(c_k-c_0)(t,m,n)$ of the $k$-th iteration $c_2(t,m,n)$.}
\label{fig4}
\end{figure}

For $\Gamma^{(k)}\ni\gamma^{(k)}=(\gamma_1^{(k-1)},\gamma_2^{(k-1)},\gamma_3^{(k-1)})
\in\Gamma^{(k-1)}\times\Gamma^{(k-1)}\times\Gamma^{(k-1)}$, define
\begin{align*}
\prod_{j=1}^3{\mathrm D}^{(k-1,\gamma_j^{(k-1)})}={\mathrm D}^{(k,\gamma^{(k)})}
\stackrel{\vec{\iota}}{\boldsymbol\simeq}
\mathscr D^{(k,\gamma^{(k)})}=\prod_{j=1}^3\mathscr D_x^{(k-1,\gamma_j^{(k-1)})}\times\prod_{j=1}^3\mathscr D_y^{(k-1,\gamma_j^{(k-1)})},
\end{align*}
where
\begin{align*}
\vec{\iota}(\spadesuit^{(k)})=~&\left(\iota_x(\spadesuit^{(k)}),\iota_y(\spadesuit^{(k)})\right)\\
=~&\left(\widetilde\spadesuit_x^{(k)},\widetilde\spadesuit_y^{(k)}\right)\\
=~&\left((\widetilde\spadesuit_{1,x}^{(k-1)},\widetilde\spadesuit_{2,x}^{(k-1)},\widetilde\spadesuit_{3,x}^{(k-1)}),
(\widetilde\spadesuit_{1,y}^{(k-1)},\widetilde\spadesuit_{2,y}^{(k-1)},\widetilde\spadesuit_{3,y}^{(k-1)})\right)\\
=~&\Big((m_j^{1})_{1\leq j\leq2\sigma(\gamma_1^{(k-1)})},(m_j^2)_{1\leq j\leq2\sigma(\gamma_2^{(k-1)})},(m_j^3)_{1\leq j\leq2\sigma(\gamma_3^{(k-1)})},\\
&\hspace{2.25mm}(n_j^1)_{1\leq j\leq2\sigma(\gamma_1^{(k-1)})},(n_j^2)_{1\leq j\leq2\sigma(\gamma_2^{(k-1)})},(n_j^3)_{1\leq j\leq2\sigma(\gamma_3^{(k-1)})}\Big).
\end{align*}

Define
\begin{align*}
\lambda(\spadesuit^{(k)})
=~&(\vec{\cc\ca\cs}\circ\vec{\iota})(\spadesuit^{(k)})\\
=~&\left((\cc\ca\cs_x\circ\iota_x)(\spadesuit^{(k)}),
(\cc\ca\cs_y\circ\iota_y)(\spadesuit^{(k)})\right)\\
=~&\left(\cc\ca\cs_x(\widetilde\spadesuit_x^{(k)}),\cc\ca\cs_y(\widetilde\spadesuit_y^{(k)})\right)\\
=~&\left(\sum_{j=1}^3(-1)^{j-1}\widetilde\spadesuit_{j,x}^{(k-1)},
\sum_{j=1}^3(-1)^{j-1}\widetilde\spadesuit_{j,y}^{(k-1)}\right)\\
=~&\sum_{j=1}^3(-1)^{j-1}\lambda(\spadesuit_j^{(k-1)}),
\end{align*}
and
\begin{align*}
\mathfrak C^{(k,\gamma^{(k)})}(\spadesuit^{(k)})&=\prod_{j=1}^3\left\{\mathfrak C^{(k-1,\gamma_j^{(k-1)})}(\spadesuit_j^{(k-1)})\right\}^{\ast^{[j-1]}};\\
\mathfrak I^{(k,\gamma^{(k)})}(t,\spadesuit^{(k)})&=
\int_0^t\Phi^{t-s}(\lambda(\spadesuit^{(k)}))\prod_{j=1}^3\left\{\mathfrak I^{(k-1,\gamma_j^{(k-1)})}(s,\spadesuit_j^{(k-1)})\right\}^{\ast^{[j-1]}}{\rm d}s;\\
\mathfrak F^{(k,\gamma^{(k)})}(\spadesuit^{(k)})&=\prod_{j=1}^3\left\{\mathfrak F^{(k-1,\gamma_j^{(k-1)})}(\spadesuit_j^{(k-1)})\right\}^{\ast^{[j-1]}}.
\end{align*}

{\em Proof of \eqref{ke}.}~~~~It follows from {\bf Step 4} that \eqref{ke} holds for $k=2$.

Let $k>2$. Assume that \eqref{ke} is true for all $2<k^\prime<k$.

Consider the nonlinear part $(c_k-c_0)(t,m,n)$ of the $k$-th iteration $c_k(t,m,n)$. By  \eqref{pi} and \eqref{c1}, we have
\begin{align}
&(c_k-c_0)(t,m,n)\nonumber\\
=~&{\rm i}\varepsilon\int_0^t\Phi^{t-s}(m,n)
\sum_{\substack{(m,n)=\sum_{j=1}^3(-1)^{j-1}(m_j,n_j)\\(m_j,n_j)\in\mathbb Z^{\nu_1}\times\mathbb Z^{\nu_2}\\j=1,2,3}}\prod_{j=1}^{3}\left\{c_{k-1}(s,m,n)\right\}^{\ast^{[j-1]}}{\rm d}s\nonumber\\
=~&{\rm i}\varepsilon\int_0^t\Phi^{t-s}(m,n)
\sum_{\substack{(m,n)=\sum_{j=1}^3(-1)^{j-1}(m_j,n_j)\\(m_j,n_j)\in\mathbb Z^{\nu_1}\times\mathbb Z^{\nu_2}\\j=1,2,3}}\prod_{j=1}^{3}\nonumber\\
&\Bigg\{
\sum_{\gamma_j^{(k-1)}\in\Gamma^{(k-1)}}
\sum_{\substack{\spadesuit_j^{(k-1)}\in{\mathrm D}^{(k-1,\gamma_j^{(k-1)})}\\\lambda(\spadesuit_j^{(k-1)})=(m_j,n_j)}}
\mathfrak C^{(k-1,\gamma_j^{(k-1)})}(\spadesuit_j^{(k-1)})\nonumber\\
&\hspace{3mm}\mathfrak I^{(k-1,\gamma_j^{(k-1)})}(s,\spadesuit_j^{(k-1)})
\mathfrak F^{(k-1,\gamma_j^{(k-1)})}(\spadesuit_j^{(k-1)})
\Bigg\}^{\ast^{[j-1]}}{\rm d}s\nonumber\\
=~&\sum_{\substack{
\gamma_j^{(k-1)}\in\Gamma^{(k-1)}\\j=1,2,3}}
\sum_{\substack{(m,n)=\sum_{j=1}^3(-1)^{j-1}(m_j,n_j)\\(m_j,n_j)\in\mathbb Z^{\nu_1}\times\mathbb Z^{\nu_2}\\j=1,2,3}}
\sum_{\substack{(m_j,n_j)=\lambda(\spadesuit_j^{(k-1)})\\(m_j,n_j)\in\mathbb Z^{\nu_1}\times\mathbb Z^{\nu_2}\\j=1,2,3}}\nonumber\\
&\prod_{j=1}^3\left\{\mathfrak C^{(k-1,\gamma_j^{(k-1)})}(\spadesuit_j^{(k-1)})\right\}^{\ast^{[j-1]}}\nonumber\\
&\times\int_0^t\Phi^{t-s}(\lambda(\spadesuit^{(k)}))
\prod_{j=1}^3\left\{\mathfrak I^{(k-1,\gamma_j^{(k-1)})}(s,\spadesuit_j^{(k-1)})\right\}^{\ast^{[j-1]}}
{\rm d}s\nonumber\\
&\times({\rm i}\varepsilon)
\prod_{j=1}^3\left\{\mathfrak F^{(k-1,\gamma_j^{(k-1)})}(\spadesuit_j^{(k-1)})\right\}^{\ast^{[j-1]}}\nonumber\\
=~&\sum_{(\gamma_1^{(k-1)},\gamma_2^{(k-1)},\gamma_3^{(k-1)})\in\Gamma^{(k-1)}\times\Gamma^{(k-1)}
\times\Gamma^{(k-1)}}\nonumber\\
&\sum_{\substack{(m,n)=\sum_{j=1}^3(-1)^{j-1}
\lambda(\spadesuit_j^{(k-1)})\\(\spadesuit_1^{(k-1)},\spadesuit_2^{(k-1)},\spadesuit_3^{(k-1)})\in{\mathrm D}^{(k-1,\gamma_1^{(k-1)})}\times{\mathrm D}^{(k-1,\gamma_2^{(k-1)})}\times{\mathrm D}^{(k-1,\gamma_3^{(k-1)})}\\\lambda(\spadesuit_j^{(k-1)})=(m_j,n_j),~j=1,2,3}}\nonumber\\
~&\prod_{j=1}^3\left\{\mathfrak C^{(k-1,\gamma_j^{(k-1)})}(\spadesuit_j^{(k-1)})\right\}^{\ast^{[j-1]}}
\nonumber\\
&\times\int_0^t\Phi^{t-s}(\lambda(\spadesuit^{(k)}))
\prod_{j=1}^3\left\{\mathfrak I^{(k-1,\gamma_j^{(k-1)})}(s,\spadesuit_j^{(k-1)})\right\}^{\ast^{[j-1]}}
{\rm d}s\nonumber\\
&\times({\rm i}\varepsilon)
\prod_{j=1}^3\left\{\mathfrak F^{(k-1,\gamma_j^{(k-1)})}(\spadesuit_j^{(k-1)})\right\}^{\ast^{[j-1]}}\nonumber\\
=~&\sum_{\gamma^{(k)}\in\Gamma^{(k)}}
\sum_{\substack{\spadesuit^{(k)}\in{\mathrm D}^{(k,\gamma^{(k)})}\\\lambda(\spadesuit^{(k)})=(m,n)}}
\mathfrak C^{(k,\gamma^{(k)})}(\spadesuit^{(k)})
\mathfrak I^{(k,\gamma^{(k)})}(t,\spadesuit^{(k)})
\mathfrak F^{(k,\gamma^{(k)})}(\spadesuit^{(k)})
\label{ken}
\end{align}

Combining \eqref{ken} with \eqref{c0l} yields
\begin{align}
&c_k(t,m,n)\nonumber\\
=~&c_0(t,m,n)+(c_k-c_0)(t,m,n)\nonumber\\
=~&\sum_{\gamma^{(k)}\in\Gamma^{(k)}}\sum_{\substack{\lambda(\spadesuit^{(k)})=(m,n)\\\spadesuit^{(k)}\in{\mathrm D}^{(k,\gamma^{(k)})}}}
\mathfrak C^{(k,\gamma^{(k)})}(\spadesuit^{(k)})\mathfrak I^{(k,\gamma^{(k)})}(t,\spadesuit^{(k)})\mathfrak F^{(1,\gamma^{(k)})}(\spadesuit^{(k)}).\label{c1}
\end{align}

This completes the analysis of the $k$-th Picard iteration.

\begin{remark}\label{re3}
Similar to the discussions in \autoref{re1} and \autoref{re2}, let $$\gamma^{(k)}=(\gamma_1^{(k-1)},\gamma_2^{(k-1)},\gamma_3^{(k-1)})\in\Gamma^{(k-1)}\times\Gamma^{(k-1)}
\times\Gamma^{(k-1)},$$
it follows from \cite{DLX1} that
\begin{align*}
\mathscr D_x^{(k,\gamma^{(k)})}=(\mathbb Z^{\nu_1})^{2\sigma(\gamma^{(k)})},\quad
\mathscr D_y^{(k,\gamma^{(k)})}=(\mathbb Z^{\nu_2})^{2\sigma(\gamma^{(k)})},
\end{align*}
where $\sigma$ stands for the first-class counting function defined by letting
\begin{align}
\sigma(\gamma^{(k)})=\sum_{j=1}^3\sigma(\gamma_j^{(k-1)}).
\end{align}
\end{remark}
By induction, we know that the combinatorial tree \eqref{ctree} holds for every $c_k(t,m,n)$, where $k\geq2$. This completes the proof of \autoref{treethm}.
\end{proof}

\subsection{The Picard Sequence is Uniform-in-Time $\vec{\kappa}/2$-Exponentially Decaying}\label{secdecay}
According to the function of $\sigma$, below we shall use $\left((\mathfrak m_j,\mathfrak n_j)\right)_{1\leq j\leq 2\sigma(\gamma^{(k)})}$ to denote the element $\spadesuit^{(k)}$ in $\mathrm D^{(k,\gamma^{(k)})}$, where $(\mathfrak m_j)_{1\leq j\leq 2\sigma(\gamma^{(k)})}\in\mathscr D_x^{(k,\gamma^{(k)})}$, $(\mathfrak n_j)_{1\leq j\leq 2\sigma(\gamma^{(k)})}\in\mathscr D_y^{(k,\gamma^{(k)})}$ and $k\geq1$. To be more precise (keep \autoref{fig0}, \autoref{fig1}, \autoref{fig3} and \autoref{fig4} in mind),
\begin{itemize}
  \item if $\gamma^{(k)}=0\in\Gamma^{(k)}$, where $k\geq1$, $\mathfrak m_1=m, \mathfrak n_1=n$;
  \item if $\gamma^{(1)}=1\in\Gamma^{(1)}$, $\mathfrak m_j=m_j, \mathfrak n_j=n_j$ for $j\in\{1,2,3\}$;
  \item if $\gamma^{(k)}=(\gamma_1^{(k-1)},\gamma_2^{(k-1)},\gamma_3^{(k-1)})\in\Gamma^{(k-1)}\times\Gamma^{(k-1)}\times\Gamma^{(k-1)}$, where $k\geq2$,
  \begin{align}\label{me}
  \mathfrak m_j=
  \begin{cases}
  m_j^1,&1\leq j\leq 2\sigma(\gamma_1^{(k-1)});\\
  m_{j-2\sigma(\gamma_1^{(k-1)})}^2,&2\sigma(\gamma_1^{(k-1)})+1\leq j\leq 2\sigma(\gamma_1^{(k-1)})+2\sigma(\gamma_2^{(k-1)});\\
  m_{j-(2\sigma(\gamma_1^{(k-1)})+2\sigma(\gamma_2^{(k-1)}))}^3,&2\sigma(\gamma_1^{(k-1)})+2\sigma(\gamma_2^{(k-1)})+1\leq j\leq2\sigma(\gamma^{(k)}),
  \end{cases}
  \end{align}
and
  \begin{align}\label{ne}
  \mathfrak n_j=
  \begin{cases}
  n_j^1,&1\leq j\leq 2\sigma(\gamma_1^{(k-1)});\\
  n_{j-2\sigma(\gamma_1^{(k-1)})}^2,&2\sigma(\gamma_1^{(k-1)})+1\leq j\leq 2\sigma(\gamma_1^{(k-1)})+2\sigma(\gamma_2^{(k-1)});\\
  n_{j-(2\sigma(\gamma_1^{(k-1)})+2\sigma(\gamma_2^{(k-1)}))}^3,&2\sigma(\gamma_1^{(k-1)})+2\sigma(\gamma_2^{(k-1)})+1\leq j\leq2\sigma(\gamma^{(k)}).
  \end{cases}
  \end{align}
\end{itemize}

\begin{lemma}\label{lemmc}
For all $k\geq1$, we have
\begin{align}\label{ce}
\mathfrak C^{(k,\gamma^{(k)})}(\spadesuit^{(k)})=\prod_{j=1}^{2\sigma(\gamma^{(k)})}\left\{c(\mathfrak m_j,\mathfrak n_j)\right\}^{\ast^{[j-1]}}.
\end{align}
Furthermore, under the condition \eqref{ic}, we have
\begin{align}\label{ece}
|\mathfrak C^{(k,\gamma^{(k)})}(\spadesuit^{(k)})|\leq e^{-\left(\kappa_1|\widetilde\spadesuit_x^{(k)}|+\kappa_2|\widetilde\spadesuit_y^{(k)}|\right)},
\end{align}
where
\begin{align*}
|\widetilde\spadesuit_x^{(k)}|=\sum_{j=1}^{2\sigma(\gamma^{(k)})}|\mathfrak m_j|\quad\text{and}\quad
|\widetilde\spadesuit_y^{(k)}|=\sum_{j=1}^{2\sigma(\gamma^{(k)})}|\mathfrak n_j|.
\end{align*}
\end{lemma}
\begin{proof}
We first use induction to prove \eqref{ce}.

Clearly, it follows from the analysis in \autoref{sectree} that \eqref{ce} holds for $\gamma^{(k)}=0\in\Gamma^{(k)}$, where $k\geq1$, and $\gamma^{(1)}=1\in\Gamma^{(1)}$. This implies that \eqref{ce} is true for $k=1$.

Let $k\geq2$. Assume that \eqref{ce} holds for all $1<k^\prime<k$.

For $\gamma^{(k)}=(\gamma_1^{(k-1)},\gamma_2^{(k-1)},\gamma_3^{(k-1)})
\in\Gamma^{(k-1)}\times\Gamma^{(k-1)}\times\Gamma^{(k-1)}$, it follows from \autoref{sectree} that
\begin{align*}
\mathfrak C^{(k,\gamma^{(k)})}(\spadesuit^{(k)})&=\prod_{j=1}^3\{\mathfrak C^{(k-1,\gamma_j^{(k-1)})}(\spadesuit_j^{(k-1)})\}^{\ast^{[j-1]}}\\
&=\prod_{j=1}^3\left\{\prod_{j^\prime=1}^{2\sigma(\gamma_{j^\prime}^{(k-1)})}\left\{c(m_j^{j^\prime},n_j^{j^\prime})\right\}^{\ast^{[j^\prime-1]}}\right\}^{\ast^{[j-1]}}\\
&=\prod_{j=1}^{2\sigma(\gamma^{(k)})}\left\{c(\mathfrak m_m,\mathfrak n_j)\right\}^{\ast^{[j-1]}}.
\end{align*}
By induction, we know that \eqref{ce} holds for all $k\geq1$.

Hence, under the condition \eqref{ic}, we have
\begin{align*}
|\mathfrak C^{(k,\gamma^{(k)})}(\spadesuit^{(k)})|&\leq \prod_{j=1}^{2\sigma(\gamma^{(k)})}|c(\mathfrak m_j,\mathfrak n_j)|\\
&\leq \prod_{j=1}^{2\sigma(\gamma^{(k)})}e^{-(\kappa_1|\mathfrak m_j|+\kappa_2|\mathfrak n_j|)}\\
&=e^{-\left(\kappa_1\sum_{j=1}^{2\sigma(\gamma^{(k)})}|\mathfrak m_j|+\kappa_2\sum_{j=1}^{2\sigma(\gamma^{(k)})}|\mathfrak n_j|\right)}\\
&=e^{-\left(\kappa_1\sum_{j=1}^3|\widetilde\spadesuit_{j,x}^{(k-1)}|+\kappa_2\sum_{j=1}^3|
\widetilde\spadesuit_{j,y}^{(k-1)}|\right)}\\
&=e^{-\left(\kappa_1|\widetilde\spadesuit_x^{(k)}|+\kappa_2|\widetilde\spadesuit_y^{(k)}|\right)}.
\end{align*}
This shows that \eqref{ece} holds for all $k\geq1$.

This completes the proof of \autoref{lemmc}.
\end{proof}

\begin{lemma}\label{lemmi}
For all $k\geq1$, we have
\begin{align}\label{ie}
|\mathfrak I^{(k,\gamma^{(k)})}(t,\spadesuit^{(k)})|\leq\frac{t^{\ell(\gamma^{(k)})}}{\mathfrak D(\gamma^{(k)})},
\end{align}
where
\begin{align*}
\ell(\gamma^{(k)})&=\sigma(\gamma^{(k)})-\frac{1}{2}\\
&=
\begin{cases}
0,&\gamma^{(k)}=0\in\Gamma^{(k)}, k\geq1;\\
1,&\gamma^{(1)}=1\in\Gamma^{(1)};\\
\sum_{j=1}^3\ell(\gamma_j^{(k-1)})+1,&
\gamma^{(k)}=(\gamma_1^{(k-1)},\gamma_2^{(k-1)},\gamma_3^{(k-1)})\\
&\in\Gamma^{(k-1)\times\Gamma^{(k-1)}\times\Gamma^{(k-1)}}, k\geq2,
\end{cases}
\end{align*}
and
\begin{align*}
\mathfrak D(\gamma^{(k)})=
\begin{cases}
1,&\gamma^{(k)}=0\in\Gamma^{(k)}, k\geq1;\\
1,&\gamma^{(1)}\in\Gamma^{(1)}\\
\ell(\gamma^{(k)})\prod_{j=1}^3\mathfrak D(\gamma_j^{(k-1)}),&\gamma^{(k)}=(\gamma_j^{(k-1)})\in(\Gamma^{(k-1)})^3.
\end{cases}
\end{align*}
\end{lemma}
\begin{proof}
According to the analysis in \autoref{sectree}, we know that \eqref{ie} holds for $\gamma^{(k)}=0\in\Gamma^{(k)}$, where $k\geq1$, and $\gamma^{(1)}=1\in\Gamma^{(1)}$. This implies that \eqref{ce} is true for $k=1$.

Let $k\geq2$. Assume that \eqref{ie} holds for all $1<k^\prime<k$.

For $\gamma^{(k)}=(\gamma_1^{(k-1)},\gamma_2^{(k-1)},\gamma_3^{(k-1)})
\in\Gamma^{(k-1)}\times\Gamma^{(k-1)}\times\Gamma^{(k-1)}$, it follows from \autoref{gen} that
\begin{align*}
|\mathfrak I^{(k,\gamma^{(k)})}(t,\spadesuit^{(k)})|&\leq
\int_0^t\prod_{j=1}^3|\mathfrak I^{(k-1,\gamma_j^{(k-1)})}(s,\spadesuit_j^{(k-1)})|{\rm d}s\\
&\leq\int_0^t\prod_{j=1}^3\frac{s^{\ell(\gamma_j^{(k-1)})}}{\mathfrak D(\gamma_j^{(k-1)})}{\rm d}s\\
&=\frac{t^{\ell(\gamma^{(k)})}}{\mathfrak D(\gamma^{(k)})}.
\end{align*}
This shows that \eqref{ie} holds for all $k\geq1$.

This completes the proof of \autoref{lemmi}.
\end{proof}
\begin{lemma}\label{lemmf}
For all $k\geq1$, we have
\begin{align}
|\mathfrak F^{(k,\gamma^{(k)})}(\spadesuit^{(k)})|\leq\varepsilon^{\ell(\gamma^{(k)})}.
\end{align}
\end{lemma}
\begin{proof}
This can be derived by induction, and the details are omitted for brevity.
\end{proof}

Putting these together yields the following result: uniform-in-time exponential decay with a worse decay rate for the Picard sequence; see \autoref{lemmcd}.
\begin{theorem}\label{lemmcd}
Let $T_\varepsilon=\mathcal O(\varepsilon^{-1})$; see \eqref{te}. For all $k\geq1$, we have
\begin{align}
|c_k(t,m,n)|\leq \mathcal A e^{-\left(\frac{\kappa_1}{2}|m|+\frac{\kappa_2}{2}|n|\right)},\quad\forall(t,m,n)\in[0,T_\epsilon]\times\mathbb Z^{\nu_1}\times\mathbb Z^{\nu_2},
\end{align}
where $\mathcal A$ is defined by \eqref{ae}.
\end{theorem}
\begin{proof}
It follows from \autoref{treethm}, \autoref{lemmc},
\autoref{lemmi} and \autoref{lemmf} that
\begin{align*}
&|c_k(t,m,n)|\nonumber\\
\leq~&\sum_{\gamma^{(k)}\in\Gamma^{(k)}}\sum_{\substack{\spadesuit^{(k)}\in\mathrm D^{(k,\gamma^{(k)})}\\\lambda(\spadesuit^{(k)})=(m,n)}}|\mathfrak C^{(k,\gamma^{(k)})}(\spadesuit^{(k)})||\mathfrak I^{(k,\gamma^{(k)})}(t,\spadesuit^{(k)})||\mathfrak F^{(k,\gamma^{(k)})}(\spadesuit^{(k)})|\\
\leq~&\sum_{\gamma^{(k)}\in\Gamma^{(k)}}\frac{(\varepsilon t)^{\ell(\gamma^{(k)})}}{\mathfrak D(\gamma^{(k)})}\sum_{\substack{\spadesuit^{(k)}\in\mathrm D^{(k,\gamma^{(k)})}\\\lambda(\spadesuit^{(k)})=(m,n)}}
e^{-\left(\frac{\kappa_1}{2}|\widetilde\spadesuit_x^{(k)}|+\frac{\kappa_2}{2}|\widetilde\spadesuit_y^{(k)}|\right)}\times
e^{-\left(\frac{\kappa_1}{2}|\widetilde\spadesuit_x^{(k)}|+\frac{\kappa_2}{2}|\widetilde\spadesuit_y^{(k)}|\right)}.
\end{align*}

On the one hand, under the condition $\lambda(\spadesuit^{(k)})=(m,n)$, keeping \autoref{fig4}, \eqref{me} and \eqref{ne} in mind, we have
\begin{align*}
e^{-\left(\frac{\kappa_1}{2}|\widetilde\spadesuit_x^{(k)}|+\frac{\kappa_2}{2}|\widetilde\spadesuit_y^{(k)}|\right)}
&=e^{-\left(\frac{\kappa_1}{2}\sum_{j=1}^{2\sigma(\gamma^{(k)})}|\mathfrak m_j|+\frac{\kappa_2}{2}\sum_{j=1}^{2\sigma(\gamma^{(k)})}|\mathfrak n_j|\right)}\\
&\leq e^{-\left(\frac{\kappa_1}{2}|\sum_{j=1}^{2\sigma(\gamma^{(k)})}(-1)^{j-1}\mathfrak m_j|+\frac{\kappa_2}{2}|\sum_{j=1}^{2\sigma(\gamma^{(k)})}(-1)^{j-1}\mathfrak n_j|\right)}\\
&=e^{-(\frac{\kappa_1}{2}|m|+\frac{\kappa_2}{2}|n|)}.
\end{align*}

On the other hand, by \cite[Lemma 9.1]{DLX24JMPA}, we have
\begin{align*}
&\sum_{\substack{\spadesuit^{(k)}\in\mathrm D^{(k,\gamma^{(k)})}\\\lambda(\spadesuit^{(k)})=(m,n)}}
e^{-\left(\frac{\kappa_1}{2}|\widetilde\spadesuit_x^{(k)}|+\frac{\kappa_2}{2}|\widetilde\spadesuit_y^{(k)}|\right)}
\\
=~&\sum_{\substack{(\mathfrak m_j,\mathfrak n_j)\in\mathbb Z^{\nu_1}\times\mathbb Z^{\nu_2}\\j=1,2,3}}e^{-\left(\frac{\kappa_1}{2}\sum_{j=1}^{2\sigma(\gamma^{(k)})}|\mathfrak m_j|+\frac{\kappa_1}{2}\sum_{j=1}^{2\sigma(\gamma^{(k)})}|\mathfrak n_j|\right)}\\
=~&\sum_{\mathfrak m_j\in\mathbb Z^{\nu_1},j=1,\cdots,2\sigma(\gamma^{(k)})}\prod_{j=1}^{2\sigma(\gamma^{(k)})}e^{-\frac{\kappa_1}{2}|\mathfrak m_j|}
\sum_{\mathfrak n_j\in\mathbb Z^{\nu_1},j=1,\cdots,2\sigma(\gamma^{(k)})}\prod_{j=1}^{2\sigma(\gamma^{(k)})}e^{-\frac{\kappa_2}{2}|\mathfrak n_j|}\\
=~&\prod_{j=1}^{2\sigma^{(k)}}\sum_{\mathfrak m_j\in\mathbb Z^{\nu_1}}e^{-\frac{\kappa_1}{2}|\mathfrak m_j|}\times\prod_{j=1}^{2\sigma^{(k)}}\sum_{\mathfrak n_j\in\mathbb Z^{\nu_2}}e^{-\frac{\kappa_1}{2}|\mathfrak n_j|}\\
\leq~&(6\kappa_1^{-1})^{2\sigma(\gamma^{(k)})\nu_1} (6\kappa_2^{-1})^{2\sigma(\gamma^{(k)})\nu_2}.
\end{align*}

With the relation $\sigma=\ell+\frac{1}{2}$ at hand, we have
\begin{align*}
|c_k(t,m,n)|
\leq (6\kappa_1^{-1})^{\nu_1}(6\kappa_2^{-1})^{\nu_2}e^{-\left(\frac{\kappa_1}{2}|m|+\frac{\kappa_2}{2}|n|\right)}
\sum_{\gamma^{(k)}\in\Gamma^{(k)}}\frac{\left((6\kappa_1^{-1})^{\nu_1}(6\kappa_2^{-1})^{\nu_2}\varepsilon t\right)^{\ell(\gamma^{(k)})}}{\mathfrak D(\gamma^{(k)})}.
\end{align*}

By \cite[Lemma 4.4]{DLX1}, we have
\begin{align*}
\sum_{\gamma^{(k)}\in\Gamma^{(k)}}\frac{\left((6\kappa_1^{-1})^{\nu_1}(6\kappa_2^{-1})^{\nu_2}\varepsilon t\right)^{\ell(\gamma^{(k)})}}{\mathfrak D(\gamma^{(k)})}\leq\frac{3}{2}
\end{align*}
provided that
\begin{align}\label{te}
0<t\leq\frac{4}{27}\left(\frac{\kappa_1}{6}\right)^{\nu_1}\left(\frac{\kappa_2}{6}\right)^{\nu_2}\varepsilon^{-1}:=T_{\varepsilon}.
\end{align}
That is,
\begin{align}\label{ck2}
|c_k(t,m,n)|\leq \mathcal Ae^{-\left(\frac{\kappa_1}{2}|m|+\frac{\kappa_2}{2}|n|\right)},\quad\forall (t,m,n)\in[0,T_\varepsilon]\times\mathbb Z^{\nu_1}\times\mathbb Z^{\nu_2},
\end{align}
where
\begin{align}\label{ae}
\mathcal A=\frac{3}{2}(6\kappa_1^{-1})^{\nu_1}(6\kappa_2^{-1})^{\nu_2}.
\end{align}
This completes the proof of \autoref{lemmcd}.
\end{proof}
\section{The Picard Sequence is a Cauchy Sequence}\label{cau}
\begin{theorem}\label{lems}
For all $k\geq1$, we have
\begin{align}\label{cke}
|(c_k-c_{k-1})(t,m,n)|\leq3^{k-1}\mathcal A^{2k+1}\varepsilon^k\frac{t^k}{k!}\sum_{\substack{(m,n)=\sum_{j=1}^{2k+1}(-1)^{j-1}(a_j,b_j)
\\(a_j,b_j)\in\mathbb Z^{\nu_1}\times\mathbb Z^{\nu_2}\\j=1,\cdots,2k+1}}e^{-\left(\frac{\kappa_1}{2}|a_j|+\frac{\kappa_2}{2}|b_j|\right)}.
\end{align}
Furthermore we have
\begin{align}\label{kds}
|(c_k-c_{k-1})(t,m,n)|\leq3^{-1}\mathcal A(6\kappa_1^{-1})^{\nu_1}(6\kappa_2^{-1})^{\nu_2}\frac{\left(3\mathcal A^2(6\kappa_1^{-1})^2(6\kappa_2^{-1})^2\varepsilon t\right)^k}{k!},\quad \forall k\geq1.
\end{align}
This implies that the Picard sequence $\{c_k(t,m,n)\}$ is a Cauchy sequence on $(t,m,n)\in[0,T_\varepsilon]\times\mathbb Z^{\nu_1}\times\mathbb Z^{\nu_2}$.
 \end{theorem}
\begin{proof}
For $k=1$, it follows from \eqref{pi} that
\begin{align*}
|(c_1-c_0)(t,m,n)|&\leq\varepsilon\int_0^t
\sum_{
\substack{
(m,n)=\sum_{j=1}^3(-1)^{j-1}(m_j,n_j)\\(m_j,n_j)\in\mathbb Z^{\nu_1}\times\mathbb Z^{\nu_2}\\j=1,2,3
}
}
\prod_{j=1}^3|c_0(s,m_j,n_j)|{\rm d}s\\
&\leq \mathcal A^3\varepsilon t\sum_{
\substack{
(m,n)=\sum_{j=1}^3(-1)^{j-1}(a_j,b_j)\\(a_j,b_j)\in\mathbb Z^{\nu_1}\times\mathbb Z^{\nu_2}\\j=1,2,3
}
}e^{-\left(\frac{\kappa_1}{2}|a_j|+\frac{\kappa_2}{2}|b_j|\right)}.
\end{align*}
This shows that \eqref{cke} holds for $k=1$.

Let $k\geq2$. Assume that \eqref{cke} is true for all $1<k^\prime<k$. For $k$, by \eqref{pi}, we have
\begin{align*}
&|(c_k-c_{k-1})(t,m,n)|\\
\leq~&
\varepsilon\int_0^t\sum_{
\substack{
(m,n)=\sum_{j=1}^3(-1)^{j-1}(m_j,n_j)\\(m_j,n_j)\in\mathbb Z^{\nu_1}\times\mathbb Z^{\nu_2}\\j=1,2,3
}
}
\\
&|\prod_{j=1}^3\left\{c_{k-1}(s,m_j,n_j)\right\}^{\ast^{[j-1]}}
-\prod_{j=1}^3\left\{c_{k-2}(s,m_j,n_j)\right\}^{\ast^{[j-1]}}|{\rm d}s\\
\leq~&\psi_1+\psi_2+\psi_3,
\end{align*}
where
\begin{align*}
\psi_1=~&\varepsilon\int_0^t\sum_{
\substack{
(m,n)=\sum_{j=1}^3(-1)^{j-1}(m_j,n_j)\\(m_j,n_j)\in\mathbb Z^{\nu_1}\times\mathbb Z^{\nu_2}\\j=1,2,3
}
}\\
~&|(c_{k-1}-c_{k-2})(s,m_1,n_1)||c_{k-1}(s,m_2,n_2)||c_{k-1}(s,m_3,n_3)|{\rm ds};\\
\psi_2=~&\varepsilon\int_0^t\sum_{
\substack{
(m,n)=\sum_{j=1}^3(-1)^{j-1}(m_j,n_j)\\(m_j,n_j)\in\mathbb Z^{\nu_1}\times\mathbb Z^{\nu_2}\\j=1,2,3
}
}\\
&|c_{k-2}(s,m_1,n_1)||(c_{k-1}-c_{k-2})(s,m_2,n_2)||c_{k-1}(s,m_3,n_3)|{\rm ds};\\
\psi_3=~&\varepsilon\int_0^t\sum_{
\substack{
(m,n)=\sum_{j=1}^3(-1)^{j-1}(m_j,n_j)\\(m_j,n_j)\in\mathbb Z^{\nu_1}\times\mathbb Z^{\nu_2}\\j=1,2,3
}
}\\
&|c_{k-1}(s,m_1,n_1)||c_{k-1}(s,m_2,n_2)||(c_{k-1}-c_{k-2})(s,m_3,n_3)|{\rm ds}.
\end{align*}

For $\psi_1$, according to the induction assumption, \eqref{ck2} and the following \autoref{fig5},
\begin{figure}[H]
\centering
\begin{tikzpicture}[scale=0.8,transform shape]
\node at (0,0.45) {$(m,n)$};
\node at (0,0) {$\bullet$};
\node at (-4.75,-0.7) {$(m_1,n_1)$};
\node at (-4,-1) {$\bullet$};
\node at (-0.75,-0.7) {$(m_2,n_2)$};
\node at (0,-1) {\color{red}$\bullet$};
\node at (2,-1) {$\bullet$};
\node at (2.75,-0.7) {$(m_3,n_3)$};
\node at (-6,-2) {$\bullet$};
\node at (-6,-2.4) {$(a_1,b_1)$};
\node at (-4,-2) {\color{gray}$\bullet$};
\node at (-4,-2.4) {$(a_j,b_j)$};
\node at (-2,-2) {$\bullet$};
\node at (-2,-2.4) {$(a_{2k-1},b_{2k-1})$};
\node at (0,-2) {\color{red}$\bullet$};
\node at (0,-2.4) {$(a_{2k},b_{2k})$};
\node at (2,-2) {$\bullet$};
\node at (2,-2.4) {$(a_{2k+1},b_{2k+1})$};
\draw (0,0)--(-4,-1);
\draw (0,0)--(0,-1);
\draw (0,0)--(2,-1);
\draw (-4,-1)--(-6,-2);
\draw (-4,-1)--(-4,-2);
\draw (-4,-1)--(-2,-2);
\draw (0,-1)--(0,-2);
\draw (2,-1)--(2,-2);

\draw [dashed] (-4.5,-1)--(2.5,-1);
\draw [dashed] (-6.5,-2)--(2.5,-2);

\end{tikzpicture}
\caption{Analysis for $\psi_1$. The black circle point $\bullet$ and the red circle point ${\color{red}\bullet}$ are assigned with the positive and negative signs respectively. The gray circle point ${\color{gray}\bullet}$ takes $\bullet$ or ${\color{red}\bullet}$ which depends on the position.}
\label{fig5}
\end{figure}
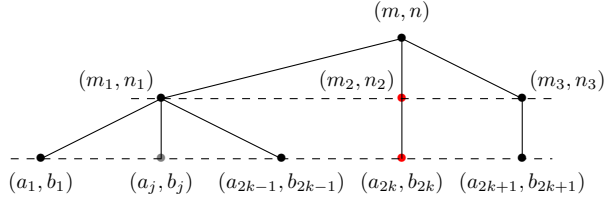
we have
\begin{align*}
\psi_1\leq~
&\varepsilon\int_0^t\sum_{
\substack{
(m,n)=\sum_{j=1}^3(-1)^{j-1}(m_j,n_j)\\(m_j,n_j)\in\mathbb Z^{\nu_1}\times\mathbb Z^{\nu_2}\\j=1,2,3
}
}3^{k-1}\mathcal A^{2k-1}\varepsilon^{k-1}\frac{s^{k-1}}{(k-1)!}\\
&
\sum_{\substack{(m_1,n_1)=\sum_{j=1}^{2k-1}(-1)^{j-1}(a_j,b_j)\\(a_j,b_j)\in\mathbb Z^{\nu_1}\times\mathbb Z^{\nu_2}\\j=1,\cdots,2k-1}}\prod_{j=1}^{2k-1}e^{-\left(\frac{\kappa_1}{2}|a_j|+\frac{\kappa_2}{2}|b_j|\right)}
\\
&\mathcal Ae^{-\left(\frac{\kappa_1}{2}|m_2|+\frac{\kappa_1}{2}|n_2|\right)}\cdot\mathcal Ae^{-\left(\frac{\kappa_1}{2}|m_3|+\frac{\kappa_2}{2}|n_3|\right)}\\
\leq~&\mathcal A^{2k+1}\varepsilon^k\frac{t^k}{k!}
\sum_{
\substack{
(m,n)=\sum_{j=1}^3(-1)^{j-1}(m_j,n_j)\\(m_j,n_j)\in\mathbb Z^{\nu_1}\times\mathbb Z^{\nu_2}\\j=1,2,3
}
}
\sum_{\substack{(m_1,n_1)=\sum_{j=1}^{2k-1}(-1)^{j-1}(a_j,b_j)\\(a_j,b_j)\in\mathbb Z^{\nu_1}\times\mathbb Z^{\nu_2}\\j=1,\cdots,2k-1}}\\
&3^{k-1}\prod_{j=1}^{2k-1}e^{-\left(\frac{\kappa_1}{2}|a_j|+\frac{\kappa_2}{2}|b_j|\right)}\cdot\mathcal Ae^{-\left(\frac{\kappa_1}{2}|m_2|+\frac{\kappa_1}{2}|n_2|\right)}\cdot\mathcal Ae^{-\left(\frac{\kappa_1}{2}|m_3|+\frac{\kappa_2}{2}|n_3|\right)}\\
=~&3^{k-1}\mathcal A^{2k+1}\varepsilon^k\frac{t^k}{k!}\sum_{\substack{(m,n)=\sum_{j=1}^{2k+1}(-1)^{j-1}(a_j,b_j)\\(a_j,b_j)\in\mathbb Z^{\nu_1}\times\mathbb Z^{\nu_2}\\j=1,\cdots,2k+1}}\prod_{j=1}^{2k+1}e^{-\left(\frac{\kappa_1}{2}|a_j|+\frac{\kappa_2}{2}|b_j|\right)}.
\end{align*}
Similarly, with the following \autoref{fig87} and \autoref{fig7},
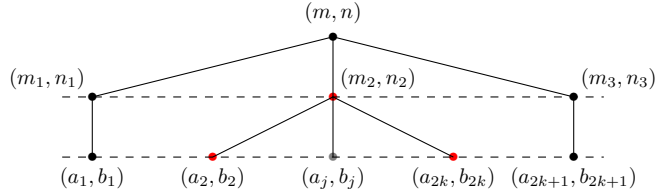
\begin{figure}[H]
\centering
\begin{tikzpicture}[scale=0.8,transform shape]
\node at (0,2.4) {$(m,n)$};
\node at (0,2){$\bullet$};
\node at (4.75,1.3) {$(m_3,n_3)$};
\node at (4,1){$\bullet$};
\node at (0,1) {\color{red}$\bullet$};
\node at (0.75,1.3) {$(m_2,n_2)$};
\node at (-4.75,1.3) {$(m_1,n_1)$};
\node at (-4,1) {$\bullet$};
\node at (-4,0) {$\bullet$};
\node at (-4,-0.35) {$(a_1,b_1)$};
\node at (-2,-0.35) {$(a_2,b_2)$};
\node at (0,-0.35) {$(a_j,b_j)$};
\node at (2,-0.35) {$(a_{2k},b_{2k})$};
\node at (4,-0.35) {$(a_{2k+1},b_{2k+1})$};
\node at (-2,0) {\color{red}$\bullet$};
\node at (0,0) {\color{gray}$\bullet$};
\node at (2,0) {\color{red}$\bullet$};
\node at (4,0) {$\bullet$};
\draw [dashed] (-4.5,1)--(4.5,1);
\draw [dashed] (-4.5,0)--(4.5,0);
\draw (0,2)--(-4,1);
\draw (0,2)--(0,1);
\draw (0,2)--(4,1);
\draw (-4,1)--(-4,0);
\draw (0,1)--(-2,0);
\draw (0,1)--(0,0);
\draw (0,1)--(2,0);
\draw (4,1)--(4,0);
\end{tikzpicture}
\caption{Analysis for $\psi_2$.}
\label{fig87}
\end{figure}
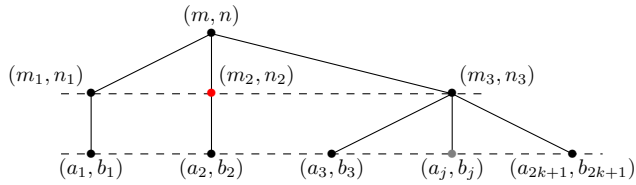
\begin{figure}[H]
\centering
\begin{tikzpicture}[scale=0.8,transform shape]
\draw (8,0)--(6,1);
\draw (6,0)--(6,1);
\draw (4,0)--(6,1);
\draw (2,0)--(2,1);

\draw (0,0)--(0,1);
\draw (2,2)--(6,1);
\draw (2,2)--(2,1);
\draw (2,2)--(0,1);
\draw [dashed](-0.5,1)--(6.5,1);
\draw [dashed](-0.5,0)--(8.5,0);
\node at (2,2.3) {$(m,n)$};
\node at (2,2) {$\bullet$};
\node at (6.75,1.3) {$(m_3,n_3)$};
\node at (6,1) {$\bullet$};
\node at (2.75,1.3) {$(m_2,n_2)$};
\node at (2,1) {\color{red}$\bullet$};
\node at (-0.75,1.3) {$(m_1,n_1)$};
\node at (0,1) {$\bullet$};
\node at (8,0) {$\bullet$};
\node at (6,0) {\color{gray}$\bullet$};
\node at (4,0) {$\bullet$};
\node at (2,0) {$\bullet$};
\node at (0,0) {$\bullet$};
\node at (0,-0.25) {$(a_1,b_1)$};
\node at (2,-0.25) {$(a_2,b_2)$};
\node at (4,-0.25) {$(a_3,b_3)$};
\node at (6,-0.25) {$(a_j,b_j)$};
\node at (8,-0.25) {$(a_{2k+1},b_{2k+1})$};
\end{tikzpicture}
\caption{Analysis for $\psi_3$.}
\label{fig7}
\end{figure}
\noindent we can prove that
\begin{align*}
\psi_j\leq&3^{k-1}\mathcal A^{2k+1}\varepsilon^k\frac{t^k}{k!}\sum_{\substack{(m,n)=\sum_{j=1}^{2k+1}(-1)^{j-1}(a_j,b_j)\\(a_j,b_j)\in\mathbb Z^{\nu_1}\times\mathbb Z^{\nu_2}\\j=1,\cdots,2k+1}}\prod_{j=1}^{2k+1}e^{-\left(\frac{\kappa_1}{2}|a_j|+\frac{\kappa_2}{2}|b_j|\right)},\quad j=1,2,3.
\end{align*}
Hence
\begin{align*}
|(c_k-c_{k-1})(t,m,n)|\leq3^{k}\mathcal A^{2k+1}\varepsilon^k\frac{t^k}{k!}\sum_{\substack{(m,n)=\sum_{j=1}^{2k+1}(-1)^{j-1}(a_j,b_j)\\(a_j,b_j)\in\mathbb Z^{\nu_1}\times\mathbb Z^{\nu_2}\\j=1,\cdots,2k+1}}\prod_{j=1}^{2k+1}e^{-\left(\frac{\kappa_1}{2}|a_j|+\frac{\kappa_2}{2}|b_j|\right)}.
\end{align*}
This shows that \eqref{cke} holds for $k$.

By induction, we know that \eqref{cke} is true for all $k\geq1$.

It follows from \eqref{cke} and \cite[Lemma 9.1]{DLX24JMPA} that \eqref{kds} is true.  Hence the Picard sequence $\{c_k(t,m,n)\}$ is a Cauchy sequence on $(t,m,n)\in[0,T_\varepsilon]\times\mathbb Z^{\nu_1}\times\mathbb Z^{\nu_2}$.

This completes the proof of \autoref{lems}.
\end{proof}

\section{Asymptotic Dynamics}\label{secasy}
Set
\begin{align}
\mathfrak c(t,m,n)&=\lim_{k\rightarrow}c_k(t,m,n), \quad (t,m,n)\in[0,T_{\varepsilon}]\times\mathbb Z^{\nu_1}\times\mathbb Z^{\nu_2};\\
\mathfrak u(t,x,y)&=\sum_{(m,n)\in\mathbb Z^{\nu_1}\times\mathbb Z^{\nu_2}}\mathfrak c(t,m,n)e^{{\rm i}(\langle m,\omega\rangle x+\langle n,\omega^\prime\rangle y)};\label{sou}\\
\mathfrak u_0(t,x,y)&=\sum_{(m,n)\in\mathbb Z^{\nu_1}\times\mathbb Z^{\nu_2}}e^{-{\rm i}\big(\langle m,\omega\rangle^2+\langle n,\omega^\prime\rangle^2\big)t}c(m,n)e^{{\rm i}(\langle m,\omega\rangle x+\langle n,\omega^\prime\rangle y)}.\label{lis}
\end{align}

We claim that $\mathfrak u$ is the unique spatially quasi-periodic solution to the Cauchy problem \eqref{NLS2}-\eqref{ID2}. The proof method for the uniqueness problem is similar to the proof of Cauchy sequence in \autoref{cau}. So we omit it here and pay attention to the asymptotic dynamical behaviour within the obtained time scale.

In fact, if $0<t\sim |\varepsilon|^{-1+\eta}$, where $0<\eta\ll1$, we have
\begin{align*}
&\|(\mathfrak u-\mathfrak u_0)(t,x,y)\|^2_{\mathcal H^{(\rho_1,\rho_2)}(\mathbb R^2)}\\
=~&\|
e^{\rho_1|m|+\rho_2|n|}
(\mathfrak c-c_0)(t,m,n)
\|^2_{\ell^2_{(m,n)}(\mathbb Z^{\nu_1}\times\mathbb Z^{\nu_2})}\\
\leq~&|\varepsilon|^2\sum_{(m,n)\in\mathbb Z^{\nu_1}\times\mathbb Z^{\nu_2}}e^{2\rho_1|m|+2\rho_2|n|}\left|\int_0^t\sum_{\substack{(m,n)=\sum_{j=1}^3(-1)^{j-1}(m_j,n_j)\\(m_j,n_j)\in\mathbb Z^{\nu_1}\times\mathbb Z^{\nu_2}}}\prod_{j=1}^3\mathfrak c(s,m_j,n_j){\rm d}s\right|^2\\
\leq~&\mathcal A^6(|\varepsilon|t)^2\sum_{(m,n)\in\mathbb Z^{\nu_1}\times\mathbb Z^{\nu_2}}e^{2\rho_1|m|+2\rho_2|n|}\\
&\left|\sum_{\substack{(m,n)=\sum_{j=1}^3(-1)^{j-1}(m_j,n_j)\\(m_j,n_j)\in\mathbb Z^{\nu_1}\times\mathbb Z^{\nu_2}\\j=1,2,3}}\prod_{j=1}^3e^{-\left(\frac{\kappa_1}{2}|m_j|+\frac{\kappa_2}{2}|n_j|
\right)}\right|^2\\
\leq~&\mathcal A^6(|\varepsilon|t)^2|
\sum_{(m,n)\in\mathbb Z^{\nu_1}\times\mathbb Z^{\nu_2}}e^{-\left(\left(\frac{\kappa_1}{2}-2\rho_1\right)|m|
+\left(\frac{\kappa_2}{2}-2\rho_2\right)|n|\right)}\\
&\left|\sum_{\substack{(m_j,n_j)\in\mathbb Z^{\nu_1}\times\mathbb Z^{\nu_2}\\j=1,2,3}}\prod_{j=1}^{3}e^{-\left(\frac{\kappa_1}{4}|m_j|+\frac{\kappa_2}{4}|n_j|\right)}
\right|^2\\
\lesssim~&_{\mathcal A,\kappa_1,\kappa_2,\rho_1,\rho_2,\nu_1,\nu_2}|\varepsilon|^{2\eta}\rightarrow0,
\end{align*}
as $\varepsilon\rightarrow0$, provided that
\begin{align*}
0<\frac{\kappa_j}{2}-2\rho_j\leq 1, \quad j=1,2.
\end{align*}

This completes the proof of \autoref{thm}.


\bibliography{34NLS}
\bibliographystyle{amsalpha}

\end{document}